\documentclass[12pt, reqno]{amsart}
\usepackage{amsmath,amssymb,amsthm}

\usepackage{lipsum}
\usepackage{amsmath, amssymb}
\usepackage{amsthm, amsfonts, mathrsfs}
\usepackage{mathptmx}
\usepackage{fullpage}
\usepackage{amsfonts,graphicx}
\numberwithin{equation}{section}
\usepackage[colorlinks=true, pdfstartview=FitV, linkcolor=blue, citecolor=blue, urlcolor=blue]{hyperref}

\newcommand\blfootnote[1]{%
  \begingroup
  \renewcommand\thefootnote{}\footnote{#1}%
  \addtocounter{footnote}{-1}%
  \endgroup
}

\newcommand{\ZZ}{\mathbb{Z}}

\newcommand{\NN}{\mathbb{N}}

\newcommand{\RR}{\mathbb{R}}

\newcommand{\EE}{\varepsilon}

\newcommand{\DD}{\textnormal{D}}

\newcommand{\Div}{\textnormal{div}}

\newcommand{\supp}{\textnormal{supp }}

\newcommand{\Curl}{\textnormal{curl}}

\newtheorem{Theo}{Theorem}[section]
\newtheorem{lem}[Theo]{Lemma}
\newtheorem{cor}[Theo]{Corollary}
\newtheorem{prop}[Theo]{Proposition}

\theoremstyle{plain}
\theoremstyle{definition}
\newtheorem{defi}[Theo]{Definition}

\theoremstyle{remark}
\newtheorem{Rema}[Theo]{Remark}
\newtheorem*{rema*}{Remark}

\author[H. Houamed]{Haroune Houamed}
\address{CNRS, LJAD, Université Co\^te d'Azur\\ Département de Mathématiques\\ Nice,  France}
\email{haroune.houamed@univ-cotedazur.fr}

\author[M. Zerguine]{Mohamed Zerguine}
\address{LEDPA, Universit\'e de Batna --2--\\ Facult\'e des Math\'ematiques et Informatique\\ D\'epartement de Math\'ematiques\\ 05000 Batna Alg\'erie}
\email{m.zerguine@univ-batna2.dz}

\begin{document}

\title[On the global solvability of the axisymmetric Boussinesq system with critical regularity]
{On the global solvability of the axisymmetric Boussinesq system with critical regularity}
\maketitle

\begin{abstract} The current paper is principally motivated by establishing the global well-posedness to the three-dimensional Boussinesq system with zero diffusivity in the setting of axisymmetric flows without swirling with $v_0\in H^{\frac12}(\RR^3)\cap \dot{B}^{0}_{3,1}(\RR^3)$ and density $\rho_0\in L^2(\RR^3)\cap \dot{B}^{0}_{3,1}(\RR^3)$. This respectively enhances the two results recently accomplished in \cite{Danchin-Paicu1, Hmidi-Rousset}. Our formalism is inspired, in particular for the first part from \cite{Abidi} concerning the axisymmetric Navier-Stokes equations once $v_0\in H^{\frac12}(\RR^3)$ and external force $f\in L^2_{loc}\big(\RR_{+};H^{\beta}(\RR^3)\big)$, with $\beta>\frac14$. This latter regularity on $f$ which is the density in our context is helpless to achieve the global estimates for Boussinesq system. This technical defect forces us to deal once again with a similar proof to that of \cite{Abidi} but with $f\in L^{\beta}_{loc}\big(\RR_{+};L^2(\RR^3))$ for some $\beta>4$. Second, we explore the gained regularity on the density by considering it as an external force in order to apply the study already obtained to the Boussinesq system. 
\end{abstract}

\noindent 
\blfootnote{{\it keywords and phrases:}
Boussinesq sytem; Navier Stokes equations; Global wellposedness; Axisymmetric solutions; Critical spaces.}  

\tableofcontents

\section{Introduction} The axisymmetric Boussinesq system in three dimensions of space with zero diffusivity is a singular perturbation at level zero of the Navier-Stokes equations. Such system reads as follows

\begin{equation}\label{B-mu}
\left\{ \begin{array}{ll} 
\partial_{t}v+v\cdot\nabla v-\mu\Delta v+\nabla p=\rho \vec e_3 & \textrm{if $(t,x)\in \RR_+\times\RR^3$,}\\
\partial_{t}\rho+v\cdot\nabla \rho=0 & \textrm{if $(t,x)\in \RR_+\times\RR^3$,}\\ 
\Div v=0, &\\ 
({v},{\rho})_{| t=0}=({v}_0,{\rho}_0).  
\end{array} \right.\tag{B$_{\mu}$}
\end{equation}
Above, $v(t,x)\in\RR^3$ refers to the velocity vector field localized in $x\in\RR^3$ at a time $t$, $p(t,x)\in\RR$ is the force of internal pressure which acts to enforce the incompressibility constraint $\Div v=0$ and it may be determined by a type of Poisson's equation and $\rho(t,x)\in\RR^{\star}_{+}$ stands either the temperature in the context of thermal convection or the mass density in the modeling of geophysical fluids. The condition $\Div v=0$, meaning that the volume of the fluid elements does not change (iso-volume) over time and $\mu>0$ is well- known as the 'kinematic' viscosity's parameter of the fluid.

\hspace{0.5cm}The key assumption of the Boussinesq system \eqref{B-mu} is that the variations in density are small and their effect is neglected everywhere except in the buoyancy term $\rho\vec e_3$ that driving the fluid motion in the direction $\vec e_3= (0,0,1) $.

\hspace{0.5cm}The Boussinesq system occurs principally in the dynamics of geophysical fluids which illustrates many models coming from atmospheric or oceanographic turbulence where rotation and stratification play an important role, see e.g. \cite{Pedlosky}.

\hspace{0.5cm}Mathematically, this model is essentially limited to the case of dimension two due to a formal resemblance (in the case $\mu=0$) with three-dimensional axisymmetric swirling flows. It can be shown that the solution develops singularities at time $t$ is related at the simultaneous blow-up of $\nabla \rho$ and the vorticity in $L^1_{t}L^\infty$, see \cite{Weinan-Shu}. Unfortunately, determining whether these quantities actually blow-up seems at least as difficult as addressing the similar problem for the system of Euler incompressible in dimension three of spaces.

\hspace{0.5cm}For better understanding the analysis of \eqref{B-mu}, we come first to the particular case which we assume that the density is constant. Thus, we obtain the classical Navier-Stokes equations

\begin{equation}\label{NS(mu)}
\left\{ \begin{array}{ll} 
\partial_{t}v+v\cdot\nabla v-\mu\Delta v+\nabla p=0& \textrm{if $(t,x)\in \RR_+\times\RR^3$,}\\ 
\Div v=0, &\\ 
v_{| t=0}=v_0.  
\end{array} \right.\tag{NS$_{\mu}$}
\end{equation}
The significative leap was J. Leray's paper \cite{Leray} in the thirties of last century where succeeded to build a weak family of solutions to \eqref{NS(mu)} globally in time in the energy space for any dimension, via compactness method. Even though, the uniqueness issue is known only in dimension two. Afterwards, H. Fujita and T. Kato formulated in \cite{Fujita-Kato} a mild-solutions for \eqref{NS(mu)} locally in time for initial data in critical Sobolev spaces $\dot H^{\frac{N}{2}-1}(\RR^N)$ which respect the so-called invariant by scaling. A similar topic are developed in several functional spaces alike $L^N(\RR^N), \dot B^{-1+\frac{N}{p}}_{p,\infty}$ and $BMO^ {-1} $, for other connected subjects we refer to \cite{Gallay-Sverak,Hmidi-Zerguine1,Houamed,Kato,Koch-Tataru, Lemarie-Rieusset,Robinson-Rodrigo-Sadowski,Planchon}. It should be noted that these kinds of solutions are globally in time apart from small initial data with respect to the viscosity parameter, except in dimension two. Thus, in general setting the global well-posedness for \eqref{NS(mu)} is a till now an open problem in PDEs. For this purpose, it is legitimate to search the "best" conditions in the sense that guarantee the existence and uniqueness solution for \eqref{NS(mu)}. This latter gives the opportunity to many authors to require that the velocity enjoying a geometric condition like {\it axisymmetric without swirl}. Namely, the velocity vector field can be written in cylindrical basis of $\RR^3$ in the following way.

\begin{equation*}
v(t, x)=v^{r}(t, r, z)\vec{e}_{r}+v^{z}(t, r, z)\vec{e}_{z},
\end{equation*}
where for every $x=(x_1,x_2,z)\in\RR^3$ we have
\begin{equation*}
x_1=r\cos\theta,\quad x_2=r\sin\theta,\quad r\ge0,\quad 0\le\theta<2\pi. 
\end{equation*}
Above, the triplet  $(\vec{e}_{r}, \vec{e}_{\theta}, \vec{e}_{z})$ represents the usual frame of unit vectors in the radial, azimutal and vertical directions with the notation
\begin{equation*}
\vec{e}_r=\Big(\frac{x_1}{r},\frac{x_2}{r},0\Big),\quad \vec{e}_{\theta}=\Big(-\frac{x_2}{r},\frac{x_1}{r},0\Big),\quad \vec{e}_{z}=(0,0,1).
\end{equation*}
Since, the vorticity in $\RR^3$ is defined by $\omega=\Curl v=\nabla\times v$, so if we aplly the $\Curl$ operator to \eqref{NS(mu)} we obtain that the vorticity takes the form $\omega\triangleq\omega_{\theta}\vec{e}_{\theta}$ with $\omega_{\theta}=\partial_{z}v^r-\partial_{r}v^z$ solving the following nonlinear parabolic equation
\begin{equation}\label{omega-theta-equation}
\partial_{t}\omega_{\theta}+(v\cdot\nabla)\omega_{\theta}-\mu\Big(\Delta\omega_{\theta}-\frac{\omega_{\theta}}{r^2}\Big)=\frac{v^r}{r}\omega_{\theta},
\end{equation}
with the notation $v\cdot\nabla=v^r\partial_r+v^z\partial_z$ and $\Delta=\partial^{2}_{r}+\frac{\partial_r}{r}+\partial_{z}^2$. Deposit $\zeta=\frac{\omega_{\theta}}{r}$, a straightforward computations lead to
\begin{equation}\label{zeta-equation}
\left\{ \begin{array}{ll} 
\partial_{t}\zeta+v\cdot\nabla \zeta-\mu\Big(\Delta +\frac{2}{r}\partial_r\Big)\zeta=0,\\
\zeta_{| t=0}=\zeta_0.
\end{array} \right.
\end{equation}
The fact that, the dissipative operator $(\Delta +\frac{2}{r}\partial_r)$ has a good sign, consequently, we may estimate $L^p-$norms of $\zeta$ globally in time, that is
\begin{equation}\label{zeta-r}
\|\zeta(t)\|_{L^p}\le \|\zeta_0\|_{L^p}, \quad p\in[1,\infty],\; t\ge0. 
\end{equation}
In fact, M. Ukhoviskii and V. Yudovich \cite{ Ukhovskii-Yudovich} independently O. Ladyzhenskaya \cite{Ladyzhenskaya} showed that the bounds \eqref{zeta-r} is considered as a bulwark to prohibit the formation of singularities in finite time for axisymmetric flows without swirl. More precisely, global existence and uniqueness were established for axisymmetric initial data $v_0\in H^1$ and $\omega_0,\frac{\omega_0}{r}\in L^2\cap L^\infty$. Lately, this latter was improved by S. Leonardi, J. M\`alek, J. Nec\u as and M. Pokorn\'y \cite{Leonardi-Malek-Necas-Pokorny} for $v_0\in H^2$ and external force $f\in L^2_{loc}(\RR_+;H^2)$ in the sense that $v_0\in H^2$ regains that $\omega_0,\frac{\omega_0}{r}\in L^2\cap L^\infty$. Fairly recent, H. Abidi \cite{Abidi} relaxed the last result by assuming that $v_0$ is in the critical Sobolev space $ H^{\frac{1}{2}}$ and $f\in L^2_{loc}(\RR_+;H^\beta)$ with $\beta>\frac14$. For critical regularities, that is $v_0\in {B}_{p, 1}^{1+\frac{3}{p}},$ with $1\le p\le \infty$ and $\frac{\omega_0}{r}\in L^{3,1}$, the second author and Hmidi investigated in \cite{Hmidi-Zerguine1} that \eqref{NS(mu)} admits a unique global solution uniformly with respect to the viscosity in the absence of Beale-Kato-Majda criterion, see \cite{Beale-Kato-Majda}. Furthermore, they also provided the rate of convergence whenever the viscosity goes to zero. 

\hspace{0.5cm}For the Boussinesq system \eqref{B-mu}, an intensive attention has been gained to the local/global well-posedness problem. In particular, in two dimensions of space and its dissipative counterpart phenomena for the density, refer to \cite{Abidi-Hmidi,Cannon-Dibenedetto,Chae,Chae-Nam0,Chae-Nam1,Danchin-Paicu0,Guo,Hmidi-Keraani0,Hmidi-Keraani1,Hmidi-Keraani-Rousset0,Hmidi-Zerguine0,Hmidi-Zerguine2,Hou-Li,Houamed-Dreyfuss,Liu-Wang-Zhang,Miao-Xue,Weinan-Shu,Zerguine} for a more complete history and references. For three-dimensional, the global regularity issue for the system \eqref{B-mu} in various spaces has received a considerable attention and widely studied in the last two decades. Worth mentioning, Danchin and Paicu showed in \cite{Danchin-Paicu1} that \eqref{B-mu} is well-posed in time for Leray's and Fujita-Kato's solutions in any dimension. Next, in axisymmetric case, Abidi-Hmidi-Keraani have been established in \cite{Abidi-Hmidi-Keraani} that \eqref{B-mu} possesses a unique global solution by rewriting it under vorticity-density formulation.

\begin{equation}\label{VD11}
\left\{ \begin{array}{ll} 
\partial_{t}\omega_{\theta}+v\cdot\nabla \omega_{\theta}-\big(\Delta-\frac{1}{r^2}\big)\omega_{\theta}=\frac{v^r}{r}\omega_{\theta}-\partial_r\rho & \\
\partial_{t}\rho+v\cdot\nabla \rho=0 & \\ (\omega_\theta,{\rho})_{| t=0}=({\omega}_0,{\rho}_0),  \end{array} \right.
\end{equation}
Consequently, the quantity $\zeta=\frac{\omega_{\theta}}{r}$ solving the equation
\begin{equation}\label{Pi-Boussinesq}
\partial_{t}\zeta+v\cdot\nabla \zeta-\big(\Delta +\frac{2}{r}\partial_r\big)\zeta=-\frac{\partial_r\rho}{r}.
\end{equation}   

They assumed that $v_0\in H^1(\RR^3),\;\xi_0\in L^{2}(\RR^3), \rho_0\in L^2\cap L^\infty$ with $\supp\rho_{0}\cap(Oz)=\emptyset$ and $P_{z}(\supp\rho_0)$ is a compact set in $\RR^3$ especially for dismissing the singularity $\frac{\partial_r\rho}{r}$, with $P_ {z} $ being the orthogonal projector over $ (Oz) $. Those results were enhanced later by Hmidi-Rousset in \cite{Hmidi-Rousset} by removing the assumption on the support of density. Their paradigm requires to assume that $(v_0,\rho_0)\in H^1(\RR^3)\times L^2(\RR^3)\cap B^ {0} _ {3,1} $ and $\frac{\omega_0}{r}\in L^2(\RR^3)$ and couples the two-equations of \eqref{VD11}  by introducing a new unknown which called {\it coupled function}.  

\hspace{0.5cm}The overall aim here is to show first that \eqref{NS(mu)} admits a unique global solution once $v_0\in H^{\frac12}(\RR^3)$ is an axisymmetric vector field and $f$ in $ L_{loc}^{\beta}\big(\RR_{+},L^2(\RR^3)\big)$, with $\beta>4$. Second, we exploit the previous result to prove that the system \eqref{B} has a unique global solution whenever $v_0\in H^{\frac12}(\RR^3)\cap\dot{B}^{0}_{3,1}(\RR^3)$ and $\rho_0\in L^{2}(\RR^3)\cap\dot{B}^{0}_{3,1}(\RR^3)$. This kind of result is considered as novel in the sense that interpolating the critical regularities following Danchin-Paicu \cite{Danchin-Paicu1} and the axisymmetric constraint of Hmidi-Rousset \cite{Hmidi-Rousset}.



\section{Statement of the main results}
\subsection{About Navier-Stokes equations with external force.}
The incompressible Navier-Stokes equations with the external force in the whole space with the viscosity $\mu=1$, still noted \eqref{NS} reads as follows.
\begin{equation}\label{NS}
\left\{ \begin{array}{ll} 
\partial_{t}v+v\cdot\nabla v-\Delta v+\nabla p=f& \textrm{if $(t,x)\in \RR_+\times\RR^3$,}\\ 
\Div v=0, &\\ 
{v}_{| t=0}={v}_0.  
\end{array} \right.\tag{NS}
\end{equation}
If $X$ is a spacial space, and $X^{(1)}$ is the space of tempered distributions $U$ such that $\nabla U$ belongs to $X$, then we denote the energy space associated to $X$ by
\begin{equation*} 
E_T(X) \triangleq L^\infty_T( X) \cap L^2_T(X^{(1)}).
\end{equation*}
The energy inequality associated to \eqref{NS} reads for all $T>0$ as follows
\begin{equation*} 
\|v\|_{L^\infty_T(L^2)} + \|\nabla v\|_{L^2_T(L^2)} \lesssim \|v_0\|_{L^2} + \|f\|_{L^2_T(\dot{H}^{-1} )}, \quad \mbox{if}\;f \in L^2_T(\dot{H}^{-1} ).
\end{equation*}
 and 
\begin{equation*}
\|v\|_{L^\infty_T(L^2)} + \|\nabla v\|_{L^2_T(L^2)} \lesssim \|v_0\|_{L^2} + (1+\sqrt{T})\|f\|_{L^2_T( {H}^{-1} )}, \quad\mbox{if} \;f\in L^2_T({H}^{-1} ).
\end{equation*}
The scenario of mild solutions will be done by rewriting \eqref{NS} in terms of a fixed point problem for a heat, semi-group, and thus obtain a unique solution $v$ which belongs to a functional space, for example $\dot H^{s}(\RR^N)$, such that the linear term $(\partial_t-\Delta)v$ and the nonlinear term ${\bf P}\Div(v\otimes v) $ have the same regularity, with ${\bf P} $ desinagtes the Leray's projector which acts on the divergence free vector field. This may involves the scaling invariance for \eqref{NS}. Indeed, if we denote by $v$ the solution of \eqref{NS} with data $v_0$ and $f$. Thus, for $\lambda>0$ the vector field
\begin{equation*}
v_{\lambda}: (t,x)\mapsto \lambda v(\lambda^2 t,\lambda x)
\end{equation*}
 with data 
\begin{equation*}
v_{0,\lambda}: x\mapsto \lambda v_0(\lambda x),\quad f_\lambda:(t,x)\mapsto \lambda^3f(\lambda^2 t,\lambda x)
\end{equation*}
is also a solution of \eqref{NS}. 


It is simple check that the Sobolev regularities according to Fujita-Kato for $N=3$ are $\dot H^{\frac12}(\RR^3)$ for $v_{0,\lambda}$, $L^2\big(\RR_+; \dot{H}^{-\frac{1}{2}}(\RR^3)\big)$ for $f_\lambda$ and $C\big(\RR_+;\dot H^{\frac12}\big)\cap L^2(\RR_+;\dot H^{\frac32}\big)$ for $v_{\lambda}$. Consequently, Fujita-Kato \cite{Fujita-Kato} succeed to recover \eqref{NS} globally in time in the following way.

\begin{prop}\label{Pr:1} Let $v_0\in \dot{H}^\frac{1}{2}(\RR^3)$ and $f \in L^2_T\big(\dot{H}^{-\frac{1}{2}}(\RR^3)\big) $, there exists some $c_0>0$ such that if
\begin{equation*} 
\|v_0\|_{\dot{H}^\frac{1}{2}}+ \|f\|_{L^{2}_{T}(\dot{H}^{-\frac{1}{2}}) }<c_0.
\end{equation*} 
Then \eqref{NS} has a unique global solution satisfying for all $T>0$,
\begin{equation*}
\|v\|_{L^\infty_T(\dot{H}^\frac{1}{2})} + \|\nabla v\|_{L^2_T(\dot{H}^\frac{1}{2})} \lesssim \|v_0\|_{\dot{H}^\frac{1}{2} } + \|f\|_{L^2_T(\dot{H}^{-\frac{1}{2}}) }.
\end{equation*}
\end{prop}
\hspace{0.5cm}Without smallness condition, Abidi \cite{Abidi} developped another result of global well-posedness in time for \eqref{NS(mu)}, by interpolating the crtical space of Fujita-Kato and the axisymmetric condition of M\`alek- Nec\u as- Pokorn\'y. Especially he proved the following theorem.
\begin{Theo}\label{Th:1} Let $v_0\in H^\frac{1}{2}(\RR^3)$ be an axisymmetric divergence free vector field vector without swirl and $f\in L^2_{loc}\big(\RR_+;H^\beta(\mathbb{R}^3)\big)$ be an axisymmetric vector field, for some $\beta >\frac{1}{4}$. Then \eqref{NS} has a unique global axisymmetric satisfying
\begin{equation*}
v\in C \big(\RR_+;H^\frac{1}{2}\big) \cap L^2_{loc}\big(\RR_+;H^\frac{3}{2}\big).
\end{equation*}
\end{Theo}
The approach suggested by Abidi to the proof of Theorem \ref{Th:1} is deeply based on argument introduced by C. P. Calder\'on \cite{Calderon} to show that the Navier-Stokes equations admit a solution in $L^p$ for $N=3, 4$ et $2<p<N$ and further performed by I. Gallagher and F. Planchon \cite{Gallagher-Planchon} with the purpose to study the global existence for the wave equations. This argument consists to split the data $v_0$ into two regular parts, the first one is a more regular, denoted by $v_{0,\ell}$ and supported spectrally in the ball of Radius $N_0$, with $N_0$ being a real number large enough, the other one is less regular but small denoted by $v_{0,h}$. Likewise, for the external $f\equiv f_{\ell}+f_{h}$, with the subscripts $h$ and $\ell$ refer respectively the high and low frequencies. More precesily, we write.
\begin{equation*}
v_{0,\ell} \triangleq \mathscr{F}^{-1}\big({\bf 1}_{|\xi|\le N_0} \widehat{v}_0\big), \quad v_{0,h} \triangleq \mathscr{F}^{-1}\big({\bf 1}_{|\xi|\ge N_0} \widehat{v }\big).
\end{equation*}
\begin{equation*}
f_{ \ell} \triangleq \mathscr{F}^{-1}\big({\bf 1}_{|\xi|\le N_0} \widehat{f}\big), \quad f_{h} \triangleq \mathscr{F}^{-1}\big({\bf 1}_{|\xi|\ge N_0} \widehat{f}\big).
\end{equation*}
Since $v_0$ and $f$ are axisymmetric, it is clear that the functions $v_{0,\ell}, v_{0,h}, f_{\ell}$ and $f_{h}$ are also because the cut-off operator doesn't disturb this structure.

\hspace{0.5cm}Compared to Proposition \ref{Pr:1} it seems that Theorem \ref{Th:1} is not optimal in terms of the required regularity on $f$, whereas the optimal regularity on $f$ would be something like $f\in L^2_T(H^{-\frac{1}{2}})$. The approach used in \cite{Abidi} isn't applicable directly if $f\in L^2_T(H^{\beta}) $, with $\beta\leq \frac{1}{4}$, instead of that, we will relax the regularity with respect to $\beta$ by means of asking for some more integrability in time. So, the general idea is as follows: for $f\in L^2_{loc}\big(\RR_+; L^2\big)$ and $\alpha\leq 0$, we have
\begin{equation*}
\|f_h\|_{L^2_T(H^{\alpha})} \le \EE_{N_0,\alpha},
\end{equation*}
with 
\begin{equation*}
\varepsilon_{N_0,\alpha} = \left\{\begin{array}{cc }
N_0^{\alpha} \|f\|_{L^2_T(L^2)} & \textrm{ if $\alpha \neq 0$}\\
\Big(\int_0^T \|{\bf 1}_{|\xi|\ge N_0}\widehat{f}(\tau,\xi)\|^2_{L^2(\RR^3)}d\tau \Big)^\frac{1}{2} & \textrm{ if $\alpha =0$.}
\end{array}
\right.
\end{equation*}

If $N_0$ is sufficiently large, then $\varepsilon_{N_0,\alpha}$ is small as much as we want, and eventually the $L^2_T(H^{\alpha})-$norm of $f_h$ as well, for all $\alpha \in \big[-\frac{1}{2}, 0\big]$. Then, for the data $f_{h}$ we define $v_h$ as the solution of the following Navier-Stokes equations
\begin{equation}\label{NS-h}
\left\{\begin{array}{l}
 \partial_t v_{h} + v_{h} \cdot \nabla v_{h} - \Delta v_{h} + \nabla p_{h}=f_h\\
  \Div v_{h} =  0 \\
  v_h{_{|t=0}} = 0.\tag{NS$_h$} 
\end{array}
\right.
\end{equation}
In the spirit of Proposition \ref{Pr:1} for \eqref{NS-h}, we will prove the following result.
\begin{prop}\label{Pr:2}
There exists $N_0>0$ large enough and $\varepsilon$ small as much as we want in terms of $N_0$ for which \eqref{NS-h} has a unique global solution $v_h$  in $E_T(\dot{H}^\frac{1}{2}) \cap  E_T(\dot{H}^1)$, with 
\begin{equation*}
\|v_h\|_{E_T(\dot{H}^\frac{1}{2}) \cap  E_T(\dot{H}^1)} \lesssim \EE.
\end{equation*}
\end{prop}

Next, we demonstrate that \eqref{NS} admits a unique global solution $v$ of the form $v=v_\ell+ v_h$ with $v_{\ell}$ satisfying the following modified system
  \begin{equation}\label{NS-l}
\left\{\begin{array}{l}
 \partial_t v_\ell + v_\ell \cdot \nabla v_\ell - \Delta v_\ell + \nabla p_\ell=f_\ell + F_{v_h}(v_\ell)\\
 \Div v_\ell =  0 \\
 v_{\ell_|t=0} =v_0,\tag{NS$_\ell$}
\end{array}\right.
  \end{equation}
where the operator $ F_{a} $ is given by 
\begin{equation}\label{the operator F_{u^h}}
F_{a}(b)  \triangleq -  a \cdot \nabla b  - b \cdot \nabla a  
\end{equation}
Unless $F_{v_h}(v_\ell)=0$, it is not clear how to solve \eqref{NS-l} by mean of the previous Proposition with an external force term as a sum of $f_\ell$ in $L^2(H^1)$ and a linear operator $F_{v_h}$. In fact, Proposition \ref{Pr:1} is not enough to establish the desired global control in time of the $H^1-$norm of the modified system \eqref{NS-l}, the issue is related to the fact that having a $L^\infty_T(H^\frac{3}{2})$ estimate of the small solution ($v_h$ in our context) is crucial to deal with the modified system. To remedy this, we will have to demand more regularity on $v_h$ which comes from additional some conditions on the external force $f$. We recall that in this work we will later treat the Boussinesq equations as well, so we want to keep the regularity zero with respect to the spacial variable, whereas we will add some integrability conditions with respect to time. Hence, we shall prove the following Proposition.

\begin{prop}\label{Pro:2}
Let $v_h$ be the solution obtained by Proposition \ref{Pr:1}, if in addition $f$ is in $L^\beta_T(L^2)$ for some $\beta>4$. Then $v_h$ is actually in $L^\infty_T(\dot{H}^\frac{3}{2})$, with
\begin{equation*}
\|v_h\|_{ L^\infty_T(\dot{H}^\frac{3}{2})} \lesssim \EE.
\end{equation*}
\end{prop}

After this extensive explanation and a stack of important results, now we are ready to state our first main result which deals by establishing the following version of Theorem \ref{Th:1}.

\begin{Theo}\label{Th:2}
Let $v_0\in H^\frac{1}{2}(\RR^3)$ be an axisymmetric divergence free vector field vector without swirl and $f\in L^\beta_{loc}\big(\RR_+;L^2(\mathbb{R}^3)\big)$ be an axisymmetric vector field vector, for some $\beta >4$. Then \eqref{NS} admits a unique global axisymmetric solution be such that 
\begin{equation*}
v\in C \big(\RR_+;H^\frac{1}{2}\big) \cap L^2_{loc}\big(\RR_+;H^\frac{3}{2}\big).
 \end{equation*}
\end{Theo} 

\begin{Rema} 
Let us point out that, the condition $\beta>4$ is essentially entailed to prove that the solution $v_h$ of \eqref{NS-h} is in $L^\infty_T(\dot{H}^\frac{3}{2})$, which is really important to get the global bound of the modified system. Equivalently, any couple of real numbers $(\alpha,\beta )\in [2,\infty]\times[-\frac{1}{2},\infty[$ such that the condition $f\in L^\beta_T(H^\alpha)$ allows to get a solution $v_h$ to \eqref{NS-h} in $E_T(H^\frac{1}{2}\cap H^1) \cap L^\infty_T(\dot{H}^\frac{3}{2}) $, is admissible to reinstate the condition $f\in L^\beta_T(L^2)$ in Theorem \ref{Th:2}.
\end{Rema}

\subsection{About the Boussinesq system} Now, let us move to the Boussineq system \eqref{B-mu} where we set the viscosity $\mu=1$ to simplify the presentation, the system often obtained still denoted by \eqref{B} and given by the following coupled equations.
\begin{equation}\label{B}
\left\{ \begin{array}{ll} 
\partial_{t}v+v\cdot\nabla v-\Delta v+\nabla p=\rho \vec e_3 & \textrm{if $(t,x)\in \RR_+\times\RR^3$,}\\
\partial_{t}\rho+v\cdot\nabla \rho=0 & \textrm{if $(t,x)\in \RR_+\times\RR^3$,}\\ 
\Div v=0, &\\ 
({v},{\rho})_{| t=0}=({v}_0,{\rho}_0).  
\end{array} \right.\tag{B}
\end{equation}
Let us recall that the system \eqref{B} respects the following scaling invariance
\begin{equation*}
 v_{\lambda} : (t,x)\mapsto \lambda v(\lambda^{2}t,\lambda x),\quad \rho_{\lambda} : (t,x)\mapsto\lambda^{3}\rho(\lambda^{2}t,\lambda x)
\end{equation*}
with initial data $(v_{0,\lambda},\rho_{0,\lambda})$ given by $v_{0,\lambda}(x)=\lambda v_{0}(\lambda x)$ and $\rho_{0,\lambda}(x)=\lambda^{3} v_{0}(\lambda x)$. In other words, the critical spaces for velocity are the same as for the Navier-Stokes system, and it is necessary to require two less derivatives on the density. A straightforward computation claims that $\dot H^{\frac12}(\RR^3)$ and $\dot{B}^{0}_{3,1}(\RR^3)$ are critical for the velocity, whereas $L^1(\RR^3)$ is critical but for the density.
 
\hspace{0.5cm}The second main result of this work handling with global well-posedness for \eqref{B} and  reads as follows.
\begin{Theo}\label{Th:3}
Let $v_0\in H^\frac{1}{2}(\RR^3)\cap \dot{B}^{0}_{3,1}(\RR^3)$  be an axisymmetric divergence free vector field vector without swirl and $\rho_0\in L^2(\RR^3) \cap \dot{B}^{0}_{3,1}(\RR^3)$ be a scalar axisymmetric function. Then \eqref{B} has a unique global solution satisfying for all $T>0$
\begin{equation*}
(v,\rho)\in E_T(H^\frac{1}{2}) \cap L^\infty_T(\dot{B}^{0}_{3,1}) \cap L^1_T(\dot{B}^{2}_{3,1}) \times  L^\infty_T(L^2 \cap\dot{B}^{0}_{3,1}). 
\end{equation*}  
\end{Theo}
A bunch of important remarks concerning the previous Theorem are in order.
\begin{Rema} 
An axisymmetric scalar function, meaning a function that depends only on the variable $ (r, z) $ but not on the angle variable $\theta$ in cylindrical coordinates. We check obviously that the axisymmetric structure is preserved through the time in the way that if $(v_0,\rho_0)$ is axisymmetric without swirl, then the obtained solution is it also.
\end{Rema}
\begin{Rema}As aforementioned above $H^{\frac12}(\RR^3)\subset\dot{H}^{\frac12}(\RR^3)$ and $\dot{B}^{0}_{3,1}(\RR^3)$ are critical spaces with respect to the velocity, whereas $L^2(\RR^3)$ and $\dot{B}^0_{3,1}(\RR^3)$ are not critical ones for the density, as pointed out, for instance, in \cite{Danchin-Paicu1} it is not clear how to solve the equations in critical spaces with respect to the density without additional required regularity on the initial data, our choice on this latest is highly inspired from \cite{Danchin-Paicu1,Hmidi-Rousset} whereas the $L^2$ condition on $\rho_0$ seems to be crucial to solve the momentum equation in $H^\frac{1}{2}$, and the condition $\dot{B}^0_{3,1}$ is essentially helpful for the uniqueness.
\end{Rema}
\hspace{0.5cm}Let us briefly discussing the proof of Theorem \ref{Th:3}, first by describing the formal energy inequalities associated to our system. To do this, let $(v,\rho,\nabla p)$ be a regular solution for Boussinesq system which decreasing at infinity with initial data $(v_0,\rho_0)$. First of all, we recall that the velocity vector field is in divergence free, so   
\begin{equation*}
\|\rho(t)\|_{L^p}=\|\rho_0\|_{L^p},\quad p\in[1,\infty],\; t\ge0. 
\end{equation*}
In particular, for all $r\in [1,\infty]$ we have
\begin{equation}\label{Eq:5}
\|\rho\|_{L^r_T(L^2)}=T^\frac{1}{r}\|\rho_0\|_{L^2}.
\end{equation} 
and thus
\begin{equation*}
\|v\|_{L^\infty_T(L^2)} + \|\nabla v\|_{L^2_T(L^2)} \lesssim \|v_0\|_{L^2} +T^{\frac12}\|\rho_0\|_{ L^2  }.
\end{equation*}
By virtue of Theorem 1.1 from \cite{Danchin-Paicu1}, the assumption $v_0,\rho_0$ in $L^2(\RR^3)$ ensures the existence of at least one global weak solution $(v,\rho)$, that is, in view of \eqref{Eq:5}, $\rho$ is in $L^\beta_T(L^2)$, for $\beta>4$. Consequently, in accordance with Theorem \ref{Th:2} we may define the unique global axisymmetric solution $v$ to the equation \eqref{NS} with $f=\rho \vec e_3$. Next step consists to propagate the regularity alike $v\in L^1_ {t} (Lip) $ which seems to be crucial, according to \cite{Danchin-Paicu1}, to prove the uniqueness of the constructed solution for \eqref{B}. But the lack of diffusion term in the density equation contributes some technical drawback. To circumvent these difficulties we will need to supply some additional summability condition at the level of the dyadic decomposition of the $(v_0,\rho_0)$, rather than being just in $ B^\frac{1}{2}_{2,2}\times B^0_ {2,2} $, namely we will need them to be in $ B^0_ {3,1} \times B^0_ {3,1} $. Theorem \ref{Th:3} is then a consequence of the following Proposition.

\begin{prop} Let $T>0$ and $(v,\rho )$ be a solution the of \eqref{B} on $(0,T)$ in $E_T(H^\frac{1}{2}) \times L^\infty_T(L^2)$. If in addition $(v_0,\rho_0) \in B^0_{3,1} \times B^0_{3,1}$, then 
\begin{equation*}
(v,\rho)\in L^\infty_T(B^0_{3,1}) \cap L^1_T(B^0_{3,1}) \times  L^\infty_T(B^0_{3,1})
\end{equation*}  
and $(v,\rho)$ is actually the unique solution of \eqref{B} on $(0,T)$.
\end{prop} 

\section{Technical tool box}During this work, we denote by $C$ a positive constant which may be different in each occurrence, but it does not depend on the initial data. We shall sometimes alternatively use the notation $X\lesssim Y$ for an inequality of type $X\le CY$ with $C$ independent of $X$ and $Y$. The notation $C_0$ means a constant depend on the involved norms of the initial data.

\hspace{0.5cm} For $s\in\RR$, the standard Sobolev spaces denoted by $H^s(\RR^3)$ is defined as the set of all tempered distributions over $\RR^3$ be such that
\begin{equation*}
\|v\|_{H^s}^{2}=\int_{\RR^3}\big(1+|\xi|^2\big)^s|\widehat{u}(\xi)|^2d\xi<\infty.
\end{equation*} 
Likewise, the homogeneous standard Sobolev spaces designated by $\dot H^s(\RR^3)$ are defined as the set of all tempered distributions over $\RR^3$ be such that $\widehat{u}\in L^1_{loc}$ and
\begin{equation*}
\|v\|_{\dot H^s}^{2}=\int_{\RR^3}|\xi|^{2s}|\widehat{u}(\xi)|^2d\xi<\infty.
\end{equation*} 
The scalar product in $H^s$ and $\dot{H}^s$ are denoted respective by $ \langle \cdot| \cdot\rangle_s $ and $\langle \cdot| \cdot\rangle_{\dot{s}}$, whereas the case $s=0$ will be simply denoted by $\langle \cdot | \cdot\rangle$.
\hspace{0.5cm}Now, we state a bref concise of Littlewood-Paley theory which motivates to define the Besov spaces, paradifferantial calculus, in particular, Bony's decompsition as well as the Bernstein's inequlaities. 

\hspace{0.5cm} Let $(\chi,\varphi)\in\mathscr{D}(\RR^3)\times \mathscr{D}(\RR^3\backslash\{0\})$ be cut-off functions, monotonically decaying along rays with values in $[0,1]$ and so that
\begin{equation*}
\supp\chi\subset \mathscr{B}(0,1),\quad\supp\varphi\subset\mathcal{A}(0,1/2,1),
\end{equation*}
with $\mathscr{B}(0,1)$ and $\mathcal{A}(0,1/2,1)$ designate respectively the unit ball and annulus with small radius $\frac12$ and big radius $1$. It easily be shown that
\begin{equation}\label{Decomp}
\forall\xi\in\RR^3,\quad \chi(\xi)+\sum_{q\ge0}\varphi(2^{-q}\xi)=1,\quad \frac{1}{2}\le\chi^2(\xi)+\sum_{q\ge0}\varphi^2(2^{-q}\xi)\le1.
\end{equation}
For every $u\in S'(\RR^3)$, define the cut-off or dyadic operators by 

\begin{equation}\label{Eq:00}
\Delta_{-1}u\triangleq\chi(\DD)u,\quad \Delta_{q}u\triangleq\varphi(2^{-q}\DD)u\quad \mbox{if}\;q\in\NN,\quad S_{q}u\triangleq\sum_{j\le q-1}\Delta_{j}u\quad\mbox{for}\; q\ge0.
\end{equation}

Besides, other nice properties $\Delta_q$ and $S_q$ are listed in the following points. Namely, for $u,v\in S'(\RR^3)$ we have
\begin{enumerate}
\item[(i)] $\vert p-q\vert\ge2\Longrightarrow\Delta_p\Delta_q u\equiv0$,
\item[(ii)] $\vert p-q\vert\ge4\Longrightarrow\Delta_q(S_{p-1}u\Delta_p v)\equiv0$,
\item[(iii)] $\Delta_q, S_q: L^p\rightarrow L^p$ uniformly with respect to  $q$ and $p$.
\item[(iv)] 
\begin{equation*}
u=\sum_{q\ge-1}\Delta_q u.
\end{equation*}
\end{enumerate}  
Likewise the homogeneous operators $\dot{\Delta}_{q}$ and $\dot{S}_{q}$ are defined by
\begin{equation}\label{Hom}
\forall{q}\in \mathbb{Z}\quad\dot{\Delta}_{q}=\varphi(2^{q}D)u, \quad \dot{S}_{q}=\sum_{ j\le q-1}\dot{\Delta}_{j}v.
\end{equation}
Now, we will recall the definition of an inhomogeneous and homogeneous Besov spaces.

\begin{defi} For $(s,p,r)\in\RR\times[1,  +\infty]^2$. The inhomogeneous Besov space $B_{p,r}^s$ (resp. the homogeneous Besov space $\dot{B}_{p,r}^s$) is the set of all tempered distributions $u\in S'$ (resp. $u\in S'_{|{\mathcal P}})$ such that
\begin{eqnarray*}
&&\Vert u\Vert_{{B}_{p, r}^{s}}\triangleq\Big(2^{qs}\Vert \Delta_{q} u\Vert_{L^{p}}\Big)_{\ell^r(\mathbb{Z})}<\infty. \\
&&\big(\mbox{resp. }\Vert u\Vert_{\dot{{B}}_{p, r}^{s}}\triangleq\ \Big(2^{qs}\Vert \dot\Delta_{q} u\Vert_{L^{p}}\Big)_{\ell^r(\mathbb{Z})}<\infty\big),
\end{eqnarray*}
where ${\mathcal P}$ stands the set of polynomials.
\end{defi}

By virtue of second estimate \eqref{Decomp}, for $s\in\RR$ that the spaces $H^s(\RR^3)$ and $\dot H^s(\RR^3)$ can be identified respectively by $B^s_{2,2}(\RR^3)$ and $\dot B^s_{2,2}(\RR^3)$ with equivalent norms. 

\hspace{0.5cm}The well-known  {\it Bony's} decomposition \cite{Bony} enables us to split formally the product of two tempered distributions $u$ and $v$ into three pieces. Especially, we have

\begin{defi} For a given $u, v\in S'$ we have
\begin{equation}\label{Eq:} 
uv=T_u v+T_v u+\mathscr{R}(u,v),
\end{equation}
with
\begin{equation*}
T_u v=\sum_{q}S_{q-1}u\Delta_q v,\quad  \mathscr{R}(u,v)=\sum_{q}\Delta_qu\widetilde\Delta_{q}v  \quad\hbox{and}\quad \widetilde\Delta_{q}=\Delta_{q-1}+\Delta_{q}+\Delta_{q+1}.
\end{equation*}
\end{defi}

The mixed space-time spaces functionals are more useful in several parts of this paper which motivates the following definition. 

\begin{defi} Let $T>0$ and $(\gamma,p,r,s)\in[1, \infty]^3\times\RR$.  We define the spaces $L^{\gamma}_{T}B_{p,r}^s$ and $\widetilde L^{\gamma}_{T}B_{p,r}^s$ respectively by: 
\begin{equation*}
L^\gamma_{T}B_{p,r}^s\triangleq\Big\{u: [0,T]\to S'; \Vert u\Vert_{L_{T}^{\gamma}B_{p, r}^{s}}=\big\Vert\big(2^{qs}\Vert \Delta_{q}u\Vert_{L^{p}}\big)_{\ell^{r}}\big\Vert_{L_{T}^{\gamma}}<\infty\Big\},
\end{equation*}
\begin{equation*}
\widetilde L^{\gamma}_{T}B_{p,r}^s\triangleq\Big\{u:[0,T]\to S'; \Vert u\Vert_{\widetilde L_{T}^{\gamma}{B}_{p, r}^{s}}=\big(2^{qs}\Vert \Delta_{q}u\Vert_{L_{T}^{\gamma}L^{p}}\big)_{\ell^{r}}<\infty\Big\}.
\end{equation*}
The relationship between these spaces is given by the following embeddings. Let $ \varepsilon>0,$ then 
\begin{equation}\label{embeddings}
\left\{\begin{array}{ll}
L^\gamma_{T}B_{p,r}^s\hookrightarrow\widetilde L^\gamma_{T}B_{p,r}^s\hookrightarrow L^\gamma_{T}B_{p,r}^{s-\varepsilon} & \textrm{if  $r\geq \beta$},\\
L^\gamma_{T}B_{p,r}^{s+\varepsilon}\hookrightarrow\widetilde L^\gamma_{T}B_{p,r}^s\hookrightarrow L^\gamma_{T}B_{p,r}^s & \textrm{if $\gamma\geq r$}.
\end{array}
\right.
\end{equation}
\end{defi}

Accordingly, we have the following interpolation result. 

\begin{cor} Let $T>0,\; s_1<s<s_2$ and $\varsigma\in(0, 1)$ such that $s=\varsigma s_1+(1-\varsigma)s_2$. Then we have
\begin{equation}\label{m1}
\Vert u\Vert_{\widetilde L_{T}^{a}{B}_{p, r}^{s}}\le C\Vert u\Vert_{\widetilde L_{T}^{a}{B}_{p, \infty}^{s_1}}^{\varsigma}\Vert u\Vert_{\widetilde L_{T}^{a}{B}_{p, \infty}^{s_2}}^{1-\varsigma}.
\end{equation}
\end{cor}

The following {\it Bernstein} inequalities describe a bound on the derivatives of a function in the $L^b-$norm in terms of the value of the function in the $L^a-$norm, under the assumption that the Fourier transform of the function is compactly supported. For more details we refer \cite{Bahouri-Chemin-Danchin,Chemin}.

\begin{lem}\label{Lem:1} There exists a constant $C>0$ such that for $1\le a\le b\le\infty$, for every function $u$ and every $q\in\NN\cup\{-1\}$, we have
\begin{equation}\label{Eq:000}
\sup_{\vert\alpha\vert=k}\Vert\partial^{\alpha}S_{q}u\Vert_{L^{b}}\le C^{k}2^{q\big(k+2\big(\frac{1}{a}-\frac{1}{b}\big)\big)}\Vert S_{q}u\Vert_{L^{a}}.
\end{equation}
\begin{equation}\label{Eq:0000}
C^{-k}2^{qk}\Vert\Delta_{q}u\Vert_{L^{a}}\le\sup_{\vert\alpha\vert=k}\Vert\partial^{\alpha}\Delta_{q}u\Vert_{L^{a}}\le C^{k}2^{qk}\Vert\Delta_{q}u\Vert_{L^{a}}.
\end{equation}
\end{lem}

A remarquable consequence of Bernstein inequality \eqref{Eq:0000} is given in the following remark. 

\begin{Rema}
\begin{equation*}
B_{p,r}^{s}\hookrightarrow B^{\widetilde s}_{\widetilde p,\widetilde r}\quad \textnormal{whenever}\; \widetilde{p}\ge p,
\end{equation*}
with
\begin{equation*}
\widetilde s<s-2\Big(\frac{1}{p}-\frac{1}{\widetilde p}\Big)\quad\textnormal{or}\quad \widetilde s=s-2\Big(\frac{1}{p}-\frac{1}{\widetilde{p}}\Big) \quad\textnormal{and}\quad \widetilde r\le r.
\end{equation*}
\end{Rema}

The following lemma gives the gain action (smoothing effects) of the dyadic block on the heat semi-group.

\begin{lem}\label{Lem:2} There exists a positive constant $C$ such that for $T\ge 0, q\ge-1$ and $p\in[1,\infty]$ we have
\begin{equation}\label{heat-kernel estimate}
 \Big\|\int_0^\infty{\bf 1}_{\{\tau\leq \cdot\}}e^{\tau\Delta} \Delta_q u(\tau)\Big\|_{L_T^{\kappa_1}(L^p)} \le C2^{2q(-1+{1}/\kappa_2-{1}/{\kappa_1})} \|\Delta_q u\|_{L^{\kappa_2}_T(L^p) }, \quad \mbox{ for all }\;\kappa_1\in [\kappa_2,\infty].
\end{equation}
\end{lem}

\begin{proof} Let $u\in L^p$ with $p\in[1,\infty]$. Since, for $q\ge-1,\;\supp\widehat{\Delta_q u}$ includes in the ball $\mathscr{B}(0,2^q)$ or in the annulus $\mathscr{A}(0,2^{q},2^{q+1})$,  then for $t\ge0$ we claim in view of convolution product and integration by parts that there exist two constants $c$ and $C$ such that
\begin{equation*}
\|e^{t\Delta}\Delta_q u(t)\|_{L^p}\le Ce^{-ct2^{2q}}\|\Delta_qe^{t\Delta}u(t)\|_{L^p}.
\end{equation*}
So, for all $\tau\in [0,t]$ we have
\begin{equation*}
 \Big\|\int_0^\infty{\bf 1}_{\{\tau\leq \cdot\}}e^{(\cdot-\tau)\Delta} \Delta_q u(\tau)\Big\|_{L^p} \le C\int_{0}^\tau e^{-c(\tau-\tau')2^{2q}}\|\Delta_q u(\tau')\|_{L^p }d\tau'. 
\end{equation*}
Via, Young's inequality with respect to time, it follows for $\kappa_1\in[\kappa_2,\infty]$ with $1/\kappa=1+1/\kappa_1-1/\kappa_2$
\begin{eqnarray*}
 \Big\|\int_0^\infty{\bf 1}_{\{\tau\leq \cdot\}}e^{(\cdot-\tau)\Delta} \Delta_q u(\tau)\Big\|_{L_T^{\kappa_1}L^p}& \le & C \Bigg(\frac{1-e^{-c\kappa_2 t2^{2q}}}{c\kappa 2^{2q}}\Bigg)^{1/\kappa}\|\Delta_q u\|_{L_T^{\kappa_2}L^p }\\
&\lesssim & 2^{2q(-1+1/\kappa_2 -1/\kappa_1)}\|\Delta_q u\|_{L_T^{\kappa_2}L^p },
\end{eqnarray*}
this completes the proof.
\end{proof}

The last result in this section deals with the logarithmic estimate in more general case first established by M. Vishik \cite{Vishik} for transport regim and devlopped later by T. Hmidi and S. Keraani \cite{Hmidi-Keraani00} for transport-diffusion one. In other words, we have. 

\begin{prop}\label{prop1} For $T>0$, let  $v$ be a divergence free vector field belonging to $L^1_T\big([0,T];\dot B^{1}_{\infty,1}\big)$. Then for $\Theta_0\in \dot B^{0}_{p,r}, \;f\in \widetilde{L}^1_{T}(\dot B^{0}_{p,r})$ the following transport-diffusion equation.
\begin{equation*}
\left\{\begin{array}{ll}
\partial_t\Theta+v\cdot\nabla\Theta-\mu\Delta\Theta=f &\\
\Theta_{t=0}=\Theta_0,
\end{array}
\right.
\end{equation*}
admits a unique solution $\Theta\in \widetilde{L}^{\infty}_{T}(\dot B^{0}_{p,r})$ satisfying for $t\in [0,T]$ 
\begin{equation*}
\|\Theta\|_{\widetilde{L}^{\infty}_{T}(\dot B^{0}_{p,r})}\le C\Big(\|\Theta_0\|_{\dot B^{0}_{p,r}}+\|f\|_{\widetilde{L}^1_{T}(\dot B^{0}_{p,r})}\Big)\Big(1+\int_{0}^{t}\|\nabla v(\tau)\|_{\dot B^{0}_{\infty,1}}d\tau\Big).
\end{equation*}
with $C$ being a universal constant.
\end{prop}

\hspace{0.5cm}Now, we will introduce some estimates between the velocity and its vorticty in the framework of axisymmetirc geometry along the following lines which are strongly useful in section \ref{First-result}. First, let us denote by $v$ the velocity vector field and $\omega$ its vorticity, with $\omega\triangleq\Curl v=\omega_{\theta}\vec e_{\theta}$. Consequently, for $x\in\RR^3$ with $r=\sqrt{x_1^2+x_2^2}\ge0$ we have $|\Omega(x)|=|\omega_{\theta}(x)|$ and 
\begin{equation*}
|\nabla\omega_{\theta}|\approx|\partial_r\omega_{\theta}(x)|+|\partial_z\omega_{\theta}(x)|+\Big|\frac{\omega_{\theta}(x)}{r}\Big|.
\end{equation*}
So, some equivalent norms beween $v$ and $\omega$ are given by the following. The proof can be found in \cite{Leonardi-Malek-Necas-Pokorny}.
\begin{lem}
Let $v$ be a smooth vector field, divergence free and axisymmetric. Then the following assertions are hold.
\begin{equation}\label{Eq:1-eqv}
\|\omega\|_{L^p}\approx\|\nabla v\|_{L^p},\quad p\in]1,\infty[.
\end{equation}
\begin{equation}\label{Eq:2-eqv}
\|D\omega_{\theta}\|_{L^p}+\Big\|\frac{\omega_{\theta}}{r}\Big\|_{L^p}\approx\|\nabla^2 v\|_{L^p}, \quad p\in]1,\infty[.
\end{equation}
\begin{equation}\label{Eq:3-eqv}
\|D^2\omega_{\theta}\|_{L^2}+\Big\|D\Big(\frac{\omega_{\theta}}{r}\Big)\Big\|_{L^2}\le C\|\nabla^3 v\|_{L^2}.
\end{equation}
The symbol $\approx$ meaning the equivalence between two norms and by $D$ we understand $(\partial_r,\partial_z)$, while $\nabla$ designates $(\partial_{x_1},\partial_{x_2},\partial_z)$.
\end{lem}

\section{Proof of the main results}
To state the proof of main results, we first need to analyze some classical estimates in Sobolev and Besov spaces that will be more useful in the proof of our main Theorems. 

We begin by proving the following proposition
\begin{prop}\label{Pr:3} For $s\in\big[\frac{1}{2}, \frac{3}{2}\big)$, there exists $C>0$ such that for $a,b$ and $c$ are regular enough, the following assertions are hold. 
\begin{equation}\label{Eq:6}
\big\langle  a\cdot \nabla c | b \big\rangle_{\dot{s}} \leq C \|a\|_{\dot{H}^{\frac{1}{2}}} ^2\|\nabla c\|_{\dot{H}^s} ^2+ \frac{1}{100}  \|\nabla b\|_{\dot{H}^s}.
\end{equation}
\begin{equation}\label{Eq:7}
\big\langle a\cdot \nabla b | b \big\rangle_{\dot{s}} \leq C \|a\|_{\dot{H}^{\frac{1}{2}}}\|\nabla b\|_{\dot{H}^s}^2
\end{equation}
 \begin{equation}\label{Eq:8}
\big\langle b\cdot \nabla a | b \big\rangle_{\dot{s}} \leq  \|\nabla a\|_{\dot{H}^{\frac{1}{2}}}^2\| b\|_{\dot{H}^s} ^2+\frac{1}{100} \|\nabla b\|_{\dot{H}^s}^2. 
 \end{equation}
\begin{equation}\label{Eq:N9}
 \big\langle a \cdot \nabla b | b  \big\rangle_{\dot s}  \leq C \| a\|_{\dot{H}^1}^4 \| b\|_{\dot{H}^s}^2  + \frac{1}{100} \|\nabla b\|_{\dot{H}^s}^2.  
 \end{equation}

For $s=\frac{3}{2}$, we have
 \begin{equation}\label{Eq:N10}
 \big\langle \Div( a\otimes b)| b \big\rangle_{\dot{\frac{3}{2}}} +\big\langle \Div( b\otimes a) | b \big\rangle_{\dot{\frac{3}{2}}} \leq C\big(\|a\|_{\dot{H}^1}^2 \|b\|_{\dot{H}^2}^2 + \|b\|_{\dot{H}^1}^2 \|a\|_{\dot{H}^2}^2\big) +  \frac{1}{100} \|\nabla b\|_{\dot{H}^\frac{3}{2}}^2.  
 \end{equation}
  And for $s= 2$, we have
 \begin{equation}\label{Eq:11}
\big\langle   a\cdot \nabla b | c \big\rangle_{\dot{2}}   \lesssim \|  a\|_{\dot{H}^1} \big(  \| \nabla b \|_{\dot{H} ^\frac{1}{2}}+ \| \nabla b \|_{\dot{H} ^2}\big) \| c\|_{\dot{H}^3}
 \end{equation}
\begin{equation}\label{Eq:12}
\big\langle   a\cdot \nabla b | c \big\rangle_{\dot{2}}   \lesssim  \big(  \| a \|_{\dot{H} ^\frac{1}{2}}+ \| a \|_{\dot{H} ^2}\big)\|   \nabla b\|_{\dot{H}^1}  \| c\|_{\dot{H}^3}
 \end{equation}
\end{prop}

\begin{proof} We restrict ourselves to establish \eqref{Eq:6}. Estimates \eqref{Eq:7} and \eqref{Eq:8} will be done by the same method. By definition we have
\begin{eqnarray*}
\big\langle  a\cdot \nabla c | b \big\rangle_{\dot{s}}^2&\le&\int_{\RR^3}|\xi|^{2s}|(\widehat{a\cdot\nabla c}(\xi))||\overline{\widehat{b}(\xi)}|d\xi\\
\nonumber&\le &\bigg(\int_{\RR^3}|\xi|^{2(s-1)}|\widehat{a\cdot\nabla c}(\xi)|^2d\xi\bigg)^{1/2}\bigg(\int_{\RR^3}|\xi|^{2(s+1)}|\widehat{b}(\xi)|^2\bigg)^{1/2}\\
&\le& \|a\cdot\nabla c\|_{\dot H^{s-1}}\|b\|_{\dot H^{s+1}}. 
\end{eqnarray*}
On the other hand, the embedding $\dot{H}^{ \frac{1}{2}}\times \dot{H}^{s}\hookrightarrow \dot{H}^{s-1}$ for $s\in\big[\frac{1}{2}, \frac{3}{2}\big)$ gives
\begin{equation*}
\|a\cdot\nabla c\|_{\dot H^{s-1}}\le C\|a\|_{\dot H^{\frac12}}\|\nabla c\|_{\dot H^{s}}.
\end{equation*}
Putting together the last two estimates to obtain the scalar product $\big\langle\cdot|\cdot\big\rangle_{\dot s}$ in $\dot H^{s}$,
\begin{equation*}
\big\langle  a\cdot \nabla c | b \big\rangle_{\dot{s}} \le C\|a\|_{\dot H^{\frac12}}\|\nabla c\|_{\dot H^{s}}\|\nabla b\|_{\dot H^{s}}.
\end{equation*}
So, Young's inequality $AB\le CA^2+\frac{1}{100} B^2$ gives the desired estimate.

\hspace{0.5cm}The proof of \eqref{Eq:N9} is ensued from the following product law, 
\begin{equation*}
\dot{H}^{s-\frac12}\times \dot{H}^1 \subset \dot{H}^{s-1} \quad \mbox{for all}\; s\in \big[\frac12,\frac32\big).
\end{equation*}
Indeed, the scalar product in $\dot{H}^s$ together with the product law above give rise to
\begin{equation*}
\big\langle  a \cdot \nabla b | b \big\rangle_{\dot{s}}   \lesssim   \|a\|_{\dot{H}^{1}}\|\nabla b\|_{\dot H^{s-\frac{1}{2}}}  \| b\|_{\dot{H}^{s+1}}
\end{equation*}
Next, in view of the following interpolation inequality in general case
\begin{equation*}
\|f\|_{\dot{H}^{s+\frac12}} \lesssim \|f\|_{\dot{H}^{s }}^\frac{1}{2}\|f\|_{\dot{H}^{s+1 }}^\frac{1}{2}
\end{equation*}
and a new use of Young inequality $AB \leq C A^4 + \frac{1}{100} B^\frac{4}{3}$ implies
\begin{equation*}
\big\langle  a \cdot \nabla b | b \big\rangle_{\dot{s} }  \leq C  \| a\|_{\dot{H}^1}^4 \|b\|_{\dot{H}^s}^2    +  \frac{1}{100} \|\nabla b\|_{\dot{H}^s}^2.  
\end{equation*}
\hspace{0.5cm}To establish \eqref{Eq:N10} for $s=\frac{3}{2}$, we apply Bony's decomposition
\begin{equation*}
 a \otimes b = T_a b + T_b a + R(a,b).
\end{equation*}
The two first terms can be estimated along the same lines as follows
\begin{eqnarray*}
\|\Delta_j T_a b \|_{L^2}&\lesssim& \|S_{j-1}a \|_{L^\infty} \|\Delta_j b \|_{L^2}\\
&\lesssim& c_j 2^{-\frac{3}{2} j}\| a\|_{\dot{B}^{-\frac{1}{2}}_{\infty,\infty}} \|b \|_{\dot{H}^2}, \quad \text{with } \sum_{j\in\ZZ} c_j^2\leq 1.
\end{eqnarray*}
Hence the Sobolev embedding $\dot{H}^1(\RR^3) \hookrightarrow  \dot{B}^{-\frac{1}{2}}_{\infty,\infty}(\RR^3)$ gives the desired estimate for $T_ab$.

For the remainder term, Bernstein inequality and the definition of Besov spaces allow to conclude
\begin{eqnarray*}
\|\Delta_j R(a,b) \|_{L^2} &\lesssim&  2^{\frac{3}{2} j}\sum_{k\geq j+N_0} \|\widetilde{\Delta}_k a \|_{L^2} \| \Delta_kb\|_{L^2} \\
&\lesssim& 2^{\frac{3}{2} j}\sum_{k\geq j+N_0} c_k  2^{-3k} \|  a \|_{\dot{H}^1} \|  b\|_{\dot{B}^2_{2,\infty}}\\
&\lesssim& 2^{-\frac{3}{2} j}\sum_{k\geq j+N_0} c_k  2^{3(j-k)} \|  a \|_{\dot{H}^1} \|  b\|_{\dot{H}^2}
\end{eqnarray*} 
thus, by using H\"older inequality, we finally obtain
\begin{equation*}
\|\Delta_j R(a,b) \|_{L^2}  \lesssim 2^{-\frac{3}{2} j} c_j \|  a \|_{\dot{H}^1} \|  b\|_{\dot{H}^2}.
\end{equation*}
\hspace{0.5cm}Finally, we only prove \eqref{Eq:11} because \eqref{Eq:12} will be done by the same fashion and we left it to the reader. The definition of the scalar product in $\dot{H}^2$ gives
\begin{equation*}
\big\langle   a\cdot \nabla b | c \big\rangle_{\dot{2}}   \lesssim \|  a\cdot \nabla b \|_{\dot{H}^1} \| c\|_{\dot{H}^3}.
\end{equation*}

Then, by using the following three-dimensional law product, see \cite{Bahouri-Chemin-Danchin}
\begin{equation*}
\dot{H}^1 \times \dot{B}^\frac{3}{2}_{2,1} \hookrightarrow \dot{H}^1
\end{equation*}
we infer that
\begin{equation*}
\big\langle   a\cdot \nabla b | c \big\rangle_{\dot{2}}   \lesssim \|  a\|_{\dot{H}^1} \| \nabla b \|_{\dot{B}^\frac{3}{2}_{2,1}} \| c\|_{\dot{H}^3}
\end{equation*}
combined with the following interpolation estimate which its proof can found again in \cite{Bahouri-Chemin-Danchin}
\begin{equation*}
\| f \|_{ \dot{B}^\frac{3}{2}_{2,1}} \lesssim \|f \|_{\dot{H}^\frac{1}{2}}+\| f \|_{\dot{H}^2},
\end{equation*}
which ends the proof of \ref{Pr:3}.
\end{proof}
The next lemma will be used to obtain the bound $L^{\infty}_{T}(\dot{H}^{\frac{3}{2}})$ of the small solution $v_h.$

\begin{lem}\label{lemma law of product1} For a regular enough divergence free vector fields $u,v$ and $\beta\in [1,\infty]$, the following estimate holds.
\begin{equation}\label{Eq:T}
\|\Div (u \otimes v )\|_{\widetilde{L}^{\beta}_{T}(\dot{B}^{0}_{2,\infty})}\lesssim \|u\|_{\widetilde{L}^{\infty}_{T} (\dot{B}^{\frac{1}{2}}_{2,\infty}) } \|v\|_{\widetilde{L}^{\beta}_{T}(\dot{B}^{2}_{2,\infty})} +\|v\|_{\widetilde{L}^{\infty}_{T} (\dot{B}^{\frac{1}{2}}_{2,\infty}) } \|u\|_{ \widetilde{L}^{\beta}_T(\dot{B}^{2}_{2,\infty})}.
\end{equation}
\end{lem}

\begin{proof} The proof is classical and relies on continuity properties of Bony's decomposition. To make presentation more reputable, we omit the tensor product and estimate the formal product $\partial_j (u v)$. Hence we need to prove 

\begin{equation*}
\|uv\|_{\widetilde{L}^{\beta}_T(\dot{B}^{1}_{2,\infty})}\lesssim \|u\|_{\widetilde{L}^\infty_T (\dot{B}^{\frac{1}{2}}_{2,\infty}) } \|v\|_{ \widetilde{L}^{\beta}_T(\dot{B}^{2}_{2,\infty})} +\|v\|_{\widetilde{L}^\infty_T (\dot{B}^{{1}/{2}}_{2,\infty}) } \|u\|_{ \widetilde{L}^{\beta}_T(\dot{B}^{2}_{2,\infty})}. 
\end{equation*}
Thanks of Bony's decomposition \eqref{Decomp}, we write
\begin{equation*}
  uv=T_{u }  v+T_{v}u +  R(u , v),
\end{equation*}
By changing $u$ and $v$'s positions, the first two terms in the r.h.s can be dealt with similarly as follows
\begin{eqnarray*}
\| \Delta_k T_{u } v \|_{L^\beta_T(L^2)}&\lesssim & \|S_{k-1} u\|_{L^{\infty}_{T}(L^\infty)} \|\Delta_k v \|_{L^\beta_T(L^2)}\\
&\lesssim & 2^{-k}\| u \|_{\widetilde{L}^\infty_T(\dot{B}^{-1}_{\infty,\infty})}\|  v \|_{\widetilde{L}^\beta_T(\dot{B}^{2}_{2,\infty})} 
\end{eqnarray*}
By Sobolev embedding we have
\begin{equation*}
\|T_{u }  v \|_{\widetilde{L}^{\beta}_T(\dot{B}^{1}_{2,\infty})}\lesssim \|u\|_{\widetilde{L}^\infty_T (\dot{B}^{\frac{1}{2}}_{2,\infty}) } \|v\|_{ \widetilde{L}^{\beta}_T(\dot{B}^{2}_{2,\infty})}.
\end{equation*}
For the reminder term, the same analysis gives
 \begin{eqnarray*}
 \|\Delta_k R(u,v) \|_{L^{\beta}_{T}(L^2)} &\lesssim& \sum_{m>k-N_0} \|\widetilde{\Delta}_m u\|_{L^\infty_T(L^\infty)} \| \Delta_m v\|_{L^\beta_T(L^2)}\\
 &\lesssim& 2^{-k} \sum_{m>k-N_0} 2^{k-m} \|u\|_{\widetilde{L}^\infty_T (\dot{B}^{{1}/{2}}_{2,\infty}) } \|v\|_{ \widetilde{L}^{\beta}_T(\dot{B}^{2}_{2,\infty})}
 \end{eqnarray*}
 thus, by Holder inequality we obtain
\begin{equation*}
\|\Delta_k R(u,v) \|_{L^\beta_T(L^2)}  \lesssim 2^{-k}  \|u\|_{\widetilde{L}^\infty_T (\dot{B}^{{1}/{2}}_{2,\infty}) } \|v\|_{ \widetilde{L}^{\beta}_T(\dot{B}^{2}_{2,\infty})}.
\end{equation*}
Lemma \ref{lemma law of product1} is then proved.
\end{proof}

\subsection{About Navier-Stokes system}\label{First-result}

\begin{proof}[Proof of Proposition \ref{Pr:2}]  Let us begin by establishing the $E_T(\dot{H}^\frac{1}{2})$ bound for $v_h$. Doing so, we take the inner product in $\dot{H}^\frac{1}{2}$ with $v_h$ solution of \eqref{NS-h} to obtain
\begin{equation*}
\frac12\frac{d}{dt} \|v_h(t)\|_{\dot{H}^\frac{1}{2}}^2 +\|\nabla v_h(t)\|_{\dot{H}^\frac{1}{2}}^2 \leq \big| \big\langle v_h\cdot \nabla v_h| v_h \big\rangle_{\dot{\frac{1}{2}}} \big| + \|f_h\|_{\dot{H}^{-\frac{1}{2}}} \|\nabla v_h\|_{\dot{H}^\frac{1}{2}}.
\end{equation*}
 By virtue of \eqref{Eq:7} in Proposition \ref{Pr:3} and Young's inequality, we infer that
\begin{equation*}
\frac12\frac{d}{dt}\|v_h(t)\|_{\dot{H}^\frac{1}{2}}^2 +\|\nabla v_h(t)\|_{\dot{H}^\frac{1}{2}}^2 \leq  C\|v_h(t)\|_{\dot{H}^\frac{1}{2}}\|\nabla v_h(t)\|_{\dot{H}^\frac{1}{2}}^2 + C\|f_h(t)\|_{\dot{H}^{-\frac{1}{2}}}^2+ \frac{1}{4} \|\nabla v_h(t)\|_{\dot{H}^\frac{1}{2}}^2.
\end{equation*} 
Now, let
\begin{equation*}
\widetilde{T}= \sup\Big\{\tau>0: \;\|v_h(\tau)\|_{\dot{H}^\frac{1}{2}} < 1/4C\Big\}.
\end{equation*}
The definition of $f_h$ provides that for all $t< \widetilde{T}$,
\begin{equation*}
\|v_h(t)\|_{\dot{H}^\frac{1}{2}}^2 +\|\nabla v_h(t)\|_{L^2_t(\dot{H}^\frac{1}{2})}^2  \leq C \EE_{N_0,-\frac{1}{2}}^2.
\end{equation*}
We choose $N_0$ large enough such that $ C\varepsilon_{N_0,-\frac{1}{2}}^2 <{1}/{(8C)^2}$, hence for all $t<\widetilde{T}$
\begin{equation}\label{contraction.estimate}
\|v_h(t)\|_{\dot{H}^\frac{1}{2}} < {1}/{8C}.  
\end{equation}
On one hand, we assume that $\widetilde{T}< \infty$, so the definition of $\widetilde{T}$ and solution's continuity yield 
\begin{equation*}
\|v_h(\widetilde{T})\|_{\dot{H}^\frac{1}{2}} = {1}/{4C}.
\end{equation*}
On the other hand, letting $t$  goes to $\widetilde{T}$ in \eqref{contraction.estimate}, we end up with
\begin{equation*}
1/4C=  \|v_h(\widetilde{T})\|_{\dot{H}^\frac{1}{2}} \leq 1/8C.
\end{equation*}
This contradicts the fact that $\widetilde{T}$ is finite, and thus $ \|v_h(t)\|_{\dot{H}^\frac{1}{2}}$ remains bounded for all $t>0$ and can not blow-up in finite time. In particular, one concludes that for all $T>0$
\begin{equation}\label{Eq:9}
\|v_h\|_{E_T(\dot{H}^\frac{1}{2})} \lesssim \EE.
\end{equation}
Let us now prove the same estimate for $\|v_h\|_{E_T(\dot{H}^1)}$. Similarly, as above, by virtue of \eqref{Eq:7} stated in Proposition \ref{Pr:3}, we write 
\begin{eqnarray*}
\frac12\frac{d}{dt}\|v_h(t)\|_{\dot{H}^1}^2+\|\nabla v_h(t)\|_{\dot{H}^1}^2 &\le &  C\|v_h(t)\|_{\dot{H}^\frac{1}{2}}\|\nabla v_h(t)\|_{\dot{H}^1}^2+ C\|f_h(t)\|_{L^2}^2+ \frac{1}{4} \|\nabla v_h(t)\|_{\dot{H}^1}^2 \\
&\le& C\|f_h(t)\|_{L^2}^2+ \frac{1}{2}\|\nabla v_h(t)\|_{\dot{H}^1}^2.
\end{eqnarray*}
Hence 
\begin{equation}\label{Eq:9.5}
\|v_h\|_{E_T(\dot{H}^1)} \lesssim \EE_{N_0,0} \lesssim \EE.
\end{equation}
Grouping and \eqref{Eq:9} and \eqref{Eq:9.5}, we end up with
\begin{equation}\label{Eq:10}
\|v_h\|_{E_T(\dot{H}^{\frac12})\cap E_T(\dot{H}^1)} \lesssim \EE.
\end{equation}
For the uniqueness topic, it is well-known to hold in these kinds of spaces, hence the details are left to the reader. The proof of  Proposition \ref{Pr:2} is now achieved.
\end{proof}


\begin{proof}[Proof of Proposition \ref{Pro:2}] The perception of the proof will be done in the following way. First, we start by establishing that $v_h\in L_{T}^\infty(B^{2-\frac{2}{\beta}}_{2,\infty})$. Thereafter, remarking that for $\beta>4$ then $2-\frac{2}{\beta}>\frac{3}{2}$, so we conclude by making use of the interpolation argument that
\begin{equation*}
L_{T}^\infty(H^1) \cap L_{T}^\infty(B^{2-\frac{2}{\beta}}_{2,\infty} )\hookrightarrow L_{T}^\infty(B^\frac{3}{2}_{2,1})\hookrightarrow L_{T}^\infty(H^\frac{3}{2} ).
\end{equation*}
So, to prove that $v_h$ relies in $ L_{T}^\infty(B^{2-\frac{2}{\lambda}}_{2,\infty})$, we first write down the equivalent Duhamel's formula associated to system \eqref{NS-h}
\begin{equation*}
v_h(t,\cdot) = -\int_0^t \mathbb{S}(t-s){\bf P}\Div(v_h\otimes v_h)(\tau,\cdot) d\tau - \int_0^t \mathbb{S}(t-\tau){\bf P} f_h(\tau,\cdot) d\tau,
\end{equation*}
where $\mathbb{S}$ represents the heat semi-group, defined by $\mathbb{S}(t) \triangleq e^{t \Delta}$.

\hspace{0.5cm}Next, we apply the Lemma \ref{Lem:2} for $a=\kappa_1, \kappa_2=2$ and $p=2$ to write 
\begin{equation}\label{heat-kernel estimate}
 \Big\| \int_{0}^{\infty} {\bf 1}_{\tau\leq \cdot}\mathbb{S}(\cdot- \tau) \Delta_q g(s)\Big\|_{L_{T}^{a}(L^2)} \lesssim 2^{2j(-1+\frac{1}{\beta}- \frac{1}{a})} \| \Delta_q g\|_{L^{\beta}_{T}(L^2) }, \quad \mbox{ for all } a\in [\beta,\infty].
\end{equation}
Accordingly, the fixed-point method for the above Duhamel's formula can be done in the following functional space.
\begin{equation*}
\mathcal{X}_T\triangleq \big\{ g\in C((0,T);L^2): \; \|g\|_{\mathcal{X}_T}<\infty \big\},
\end{equation*}
equipped with the norm
\begin{equation*}
\|g\|_{\mathcal{X}_T} \triangleq \|g\|_{\widetilde{L}^\infty_T(\dot{B}^{2-\frac{2}{\beta}}_{2,\infty})}+ \|g\|_{\widetilde{L}^{\beta}_T(\dot{B}^{2}_{2,\infty})}.
\end{equation*}
The external term can be estimated by applying \eqref{heat-kernel estimate}, while the bilinear form can be carried out by employing \eqref{Eq:T} in Lemma \ref{lemma law of product1}, we conclude that
\begin{equation}\label{eq1}
\|v_h\|_{\mathcal{X}_T} \leq C \|v_h\|_{ \widetilde{L}^{\infty}_T (\dot{B}^{\frac{1}{2}}_{2,\infty}) } \|v_h\|_{\mathcal{X}_T} + C \|f_h\|_{\widetilde{L}^{\beta}_T(\dot{B}^{0}_{2,\infty}) }. 
\end{equation}
The fact that $L^2=\dot B^{0}_{2,2}\hookrightarrow \dot{B}^{0}_{2,\infty}$, see Lemma \ref{Lem:1} and the first embedding of \eqref{embeddings} leading to
\begin{equation*}
\|f_h\|_{\widetilde{L}^{\beta}_T(\dot{B}^{0}_{2,\infty}) } \leq \|f_h\|_{L^\beta_T(L^2)}.
\end{equation*}
Similarly, on account of Proposition \ref{Pr:2}, one may write
\begin{equation*}
\|v_h\|_{ \widetilde{L}^\infty_T (\dot{B}^{\frac{1}{2}}_{2,\infty}) } \leq \|v_h\|_{L^\infty_T (\dot{H}^{\frac{1}{2}} ) } \leq \varepsilon/C,
\end{equation*}
plug these into \eqref{eq1} gives
\begin{equation*}
\|v_h\|_{\mathcal{X}_T} \lesssim \|f_h\|_{L^{\beta}_T(L^2)}.
\end{equation*}
By definition of $f_h$, for $N_0$ large enough we obtain the desired estimate.
\end{proof}

\hspace{0.5cm}Before to lay down the proof of Theorem \ref{Th:2} we begin by proving the following result.

\begin{Theo}\label{thm.NS-l with regular initial data} Let $v_0$ be an axisymmetric divergence free vector field in $H^2(\RR^3)$ and $v_h$ the unique global solution of \eqref{NS-h} given by Propositions \ref{Pr:2} and \ref{Pro:2}. Then \eqref{NS-l} has a unique solution in $E_T(H^1)$ for all $T>0$.
\end{Theo}
The proof of Theorem \ref{thm.NS-l with regular initial data} will be orchestrate, first by showing that \eqref{NS-l} has a unique local solution in $E_{T^\star}(H^s)$, for all $s\in[\frac{1}{2}, 2]$, with $T^\star$ can be considred as the maximal lifespan in $H^1$. Second, we explore the axisymmetric structure to control the $H^1-$norm over $(0,T^\star]$, which allows to extend the solution beyond $T^\star$. This  contradicts the fact that $T^\star$ is maximal.

\hspace{0.5cm}We start by proving the following proposition.
\begin{prop}\label{prop.local.NS_l} Let $v_0$ be a divergence free vector field in $H^2(\RR^3)$ and $v_h$ the unique global solution of \eqref{NS-h} oft constructed in Propositions \ref{Pr:2} and \ref{Pro:2}. Then, there exists a maximal lifespan $T^\star>0$ such that \eqref{NS-l} has a unique solution in $E_{T^\star}(\dot{H}^s)$ for all $s\in[\frac{1}{2},2]$.
\end{prop}

\begin{proof} We will only show the exsitence part and we skip the uniqueness, since this latter is well-known to hold in such spaces (the linear part with respect to $v_\ell$ can be dealt with as in the existence part). Define the free solution as  
\begin{equation*}
v^{l}(t,\cdot) \triangleq e^{t\Delta}v_0(\cdot)
\end{equation*}
and we look for a solution $v^\ell$ to \eqref{NS-l} of the form $v^\ell= v^l+ w^\ell$ with $w^\ell$ solving
  \begin{equation}\label{NS-ll}
\left\{\begin{array}{l}
 \partial_t w^\ell + w^\ell \cdot \nabla w^\ell - \Delta w^\ell + \nabla p^\ell=f_\ell + F_{v_h}(w^\ell) + F_{v^l}(w^\ell) - v^l\cdot \nabla v^l\\
\Div v =  0 \\
   w^\ell_{|t=0} = 0,\tag{NS$^\ell$}
\end{array}\right.
  \end{equation}
with $F_a(b) = - a\cdot \nabla b - b\cdot \nabla a$.

\hspace{0.5cm}Let us point out first that, for all $s\in[\frac{1}{2}, 2]$ and $p\in [2,\infty]$, we have
\begin{equation*}
\|v^l\|_{L^p_T(\dot{H}^{s+\frac{2}{p}} )} \lesssim \|v_0\|_{\dot{H}^s}.
\end{equation*}
We only restrict ourselves to outline some a priori estimates while the construction of a solution to \eqref{NS-ll}, one may employ the well-known Friedrichs's method based on formal calculation that we could prove it here. For more details about this method, we refer the reader to \cite{Bahouri-Chemin-Danchin}.

\hspace{0.5cm}We proceed in four steps.\\
$\bullet$ $ 1^{st}$ {\it Step:} $s=\frac{1}{2}$. Taking the scalar product in $\dot{H}^{\frac{1}{2}}$ for \eqref{NS-ll} with $w^\ell$. Thus, \eqref{Eq:6} in Proposition \ref{Pr:3} gives 
\begin{eqnarray}
\frac{1}{2} \|w^\ell(t)\|_{ \dot{H}^\frac{1}{2}}^2 +\|\nabla w^\ell \|_{ L^2_t(\dot{H}^\frac{1}{2})}^2  &\leq & \Big(C \|w^\ell\|_{ L^\infty_t(\dot{H}^\frac{1}{2})} + C\|v_h\|_{L^\infty_t(\dot{H}^\frac{1}{2})} + {1}/{4}\Big)\|\nabla w^\ell\|_{ L^2_t(\dot{H}^\frac{1}{2})}^2 \label{eq-w^l} \nonumber \\&&
+C\|w^\ell\|_{L^\infty_t(\dot{H}^\frac{1}{2})}^2 \|\nabla v_h\|_{L^2_t(\dot{H}^\frac{1}{2})}^2 +C \delta(t),
\end{eqnarray}
with 
\begin{equation*}
\delta(t) \triangleq  \|v^l\|_{L^\infty_t(\dot{H}^\frac{1}{2})}^2 \|\nabla v^l\|_{L^2_t(\dot{H}^\frac{1}{2})}^2+ \|v^l\|_{L^4_t(\dot{H}^1)}^4 \|w^\ell\|_{L^\infty_t(\dot{H}^\frac{1}{2})}^2 + \| f_\ell\|_{L^2_t(\dot{H}^{-\frac{1}{2}})}^2 .
\end{equation*}
According to Proposition \ref{Pr:2}, we may choose $N_0$ large enough so that 
\begin{equation*}
\|v_h\|_{L^\infty_t(\dot{H}^\frac{1}{2})}+\|\nabla v_h\|_{L^2_t(\dot{H}^\frac{1}{2})} \leq \min\big\{ {1}/{4C}, {1}/{4\sqrt{C}}\big\},
\end{equation*}
this gives in particular 
\begin{equation*}
C\|w^\ell\|_{L^\infty_t(\dot{H}^\frac{1}{2})}^2 \|\nabla v_h\|_{L^2_t(\dot{H}^\frac{1}{2})}^2 \leq \frac14\|w^\ell\|_{L^\infty_t(\dot{H}^\frac{1}{2})}^2.
\end{equation*}
By denoting 
\begin{equation*}
T_1\triangleq \sup\Big \{ \tau>0: C\|w^\ell\|_{L^\infty_\tau(\dot{H}^\frac{1}{2})} <{1}/{4} \Big \}.
\end{equation*}
Inequality \eqref{eq-w^l} gives, for all $t<T_1$
\begin{equation*}
\frac{1}{4} \|w^\ell(t)\|_{ \dot{H}^\frac{1}{2}}^2 \leq C \delta(t)
\end{equation*}
and seen that $\|\nabla v^l(\cdot)\|_{\dot{H}^\frac{1}{2}} \in L^2(\RR_+)$ and $\|v^l(\cdot)\|_{\dot{H}^1} \in L^4(\RR_+)$, one has
\begin{equation*} 
\delta(t)\longrightarrow 0,\quad t\rightarrow 0^+.
\end{equation*}
Hence, if we denote  $T_2 $ the first $t>0$ such that  $\delta(t) < {1}/{64 C^2} $ for all $t<T_2$. Setting $T^\star\triangleq \min\{T_1,T_2\}$, by usual continuity arguments, we conclude that, for all $t\in [0,T^\star)$, the $E_{t}(\dot{H}^\frac{1}{2})$ norm of $w^\ell$ remains bounded for all $t<T^\star$, more precisely we have.
\begin{equation*}
\|w^\ell\|_{L^\infty_t(\dot{H}^\frac{1}{2})}  <{1}/{4C}.
\end{equation*}
{$\bullet$ $2^{nd}$ {\it Step: }$ s\in (\frac{1}{2}, \frac{3}{2})$}. Again the scalar product in $\dot{H}^{s}$ for \eqref{NS-ll} with $w^\ell$ and \eqref{Eq:7} and \eqref{Eq:8} of Proposition \ref{Pr:3} allow us to get 
\begin{eqnarray*}
\frac{1}{2} \|w^\ell(t)\|_{ \dot{H}^s}^2 + \|\nabla w^\ell \|_{ L^2_t(\dot{H}^s)}^2  &\le& \Big(C \|w^\ell\|_{ L^\infty_t(\dot{H}^\frac{1}{2})} + C\|v_h\|_{L^\infty_t(\dot{H}^\frac{1}{2})} + {1}/{4}\Big)\|\nabla w^\ell\|_{ L^2_t(\dot{H}^s)}^2 \\
&&+C\|w^\ell\|_{L^\infty_t(\dot{H}^s)}^2 \bigg( \|\nabla v_h\|_{L^2_t(\dot{H}^\frac{1}{2})}^2 + \|v^l \|_{ L^4_t(\dot{H}^1)}^4\bigg)+C  \delta_s(t),
\end{eqnarray*}
with 
\begin{equation*}
\delta_s(t) \triangleq \|w^\ell\|_{ L^4_t(\dot{H}^1)}^2 \|v^l\|_{ L^4_t(\dot{H}^{s+\frac{1}{2}})}^2 +   \| f_\ell\|_{L^2_t(\dot{H}^{s-1})}^2 < \infty, \quad \forall t<T^\star.
\end{equation*}
We may suppose that, for $t$ small enough (we keep the same notation for the maximal time $T^\star$) 
\begin{equation*}
\|u^l \|_{ L^4_t(\dot{H}^1)}^4 < {1}/{4C}
\end{equation*}
to end with the following estimate, for all $t<T^\star$
\begin{equation*}
\|w^\ell(t)\|_{ \dot{H}^s}^2 + \|\nabla w^\ell \|_{ L^2_t(\dot{H}^s)}^2 \lesssim \delta_s(t)< \infty. 
\end{equation*}
$\bullet$ $3^{rd}$  {\it Step:} $s=\frac{3}{2}$. By similar arguments and by making use of inequality \eqref{Eq:N10}, it holds
\begin{eqnarray*}
\frac{1}{2} \|w^\ell(t)\|_{ \dot{H}^\frac{3}{2}}^2 + \|\nabla w^\ell\|_{ L^2_t(\dot{H}^\frac{3}{2})}^2  &\le&  C  \widetilde{\delta}(t) +\frac{1}{4}\|\nabla w^\ell \|_{ L^2_t(\dot{H}^s)}^2, 
\end{eqnarray*}
with 
\begin{align*}
\widetilde{\delta}(t) &\triangleq   \|v^l\|_{L^\infty_t(\dot{H}^1)}^2\|v^l\|_{L^2_t(\dot{H}^2)}^2 +\|w^\ell\|_{L^\infty_t(\dot{H}^1)}^2\|w^\ell\|_{L^2_t(\dot{H}^2)}^2 + \|v^l\|_{L^\infty_t(\dot{H}^1)}^2\|w^\ell\|_{L^2_t(\dot{H}^2)}^2 \\
& +\|w^\ell\|_{L^\infty_t(\dot{H}^1)}^2\|v^l\|_{L^2_t(\dot{H}^2)}^2 +\|v_h\|_{L^\infty_t(\dot{H}^1)}^2\|w^\ell\|_{L^2_t(\dot{H}^2)}^2  + \|w^\ell\|_{L^\infty_t(\dot{H}^1)}^2\|v_h\|_{L^2_t(\dot{H}^2)}^2 + \| f_\ell\|_{L^2_t(\dot{H}^{ \frac{1}{2}})}^2. 
\end{align*}
One deduces for all $t<T^{\star}$ that 
\begin{equation*} 
\|w^\ell(t)\|_{ \dot{H}^\frac{3}{2}}^2+\|\nabla w^\ell \|_{ L^2_t(\dot{H}^\frac{3}{2})}^2   \lesssim   \widetilde{\delta}(t) <\infty.
\end{equation*}
$\bullet$ $4^{th}$ {\it Step} $s=2$. By using estimates \eqref{Eq:11} and \eqref{Eq:12}, similar calculations to the previous steps lead to
\begin{equation}
\frac{1}{2} \|w^\ell(t) \|_{\dot{H}^2}^2  + \|\nabla w^\ell(t) \|_{L^2_t(\dot{H}^2)}^2\leq C\overline{\delta}(t) + C \int_0^t\bigg( \| \nabla v^h (\tau)\|_{ \dot{H }^1 }^2 \|w^\ell(\tau) \|_{\dot{H}^2}^2 \bigg)  d\tau+ \frac{1}{4} \|\nabla w^\ell(t) \|_{L^2_t(\dot{H}^2)}^2 
\end{equation}
with
\begin{align*}
\overline{\delta}(t) &\triangleq  \bigg(  \|v^h \|_{L^4_t(\dot{H }^1)}^2+  \|w^\ell \|_{L^4_t(\dot{H }^1)}^2 + \|v^l \|_{L^4_t(\dot{H }^1)}^2 \bigg) \bigg( \| \nabla w^\ell \|_{L^4_t(\dot{H}^\frac{1}{2} )}^2 + \| \nabla w^\ell \|_{L^4_t(\dot{H}^2)}^2 \bigg)\\
&+\bigg(  \|w^\ell \|_{L^4_t(\dot{H }^1)}^2  +  \|w^\ell \|_{L^4_t(\dot{H }^1)}^2  \bigg) \bigg( \| \nabla v^l \|_{L^4_t(\dot{H}^\frac{1}{2} )}^2 + \| \nabla v^l \|_{L^4_t(\dot{H}^2)}^2 \bigg) \\
&+  \| \nabla v^h \|_{L^2_t(\dot{H }^1)}^2 \| \nabla w^\ell \|_{L^\infty _t(\dot{H}^\frac{1}{2} )}^2 + \| f_\ell\|_{L^2_t(\dot{H}^{ 1})}^2,
\end{align*}
according to the previous steps $\overline{\delta}(t)$ is finite for all $t<T^*$, Gronwall lemma then insures 
\begin{equation*}
 \|w^\ell(t) \|_{\dot{H}^2}^2  + \|\nabla w^\ell(t)\|_{L^2_t(\dot{H}^2)}^2 \lesssim \overline{\delta}(t) \exp\big(C \|  v^h \|_{E_t(\dot{H }^1)}^2\big) <\infty ,\quad \forall t<T^*
\end{equation*}
This completes the proof of Proposition \ref{prop.local.NS_l}.
\end{proof}


\begin{proof}[Proof of Theorem \ref{thm.NS-l with regular initial data}]
To prove theorem \ref{thm.NS-l with regular initial data}, we consider the unique local solution $v_\ell$ to \eqref{NS-l} given by Proposition \ref{prop.local.NS_l} in $E_{T}(H^s)$ for all $T<T^\star$ and $s\in[\frac{1}{2},2]$, with $T^\star$ can be considred as the maximal lifespan in $H^1$, and we show now how to benefit of the axisymmetric structure to control the $H^1$ norm on $[T^{\star}-\eta,T^{\star}]$ for some $\eta<<1$, which allows to extend the solution, and contredicts the fact that $T^{\star}$ is maximal.

\hspace{0.5cm}The key idea to prove the $H^1$ global bound of $v_\ell$ is to explore the axisymmetric structure to get a uniform bound of $\alpha \triangleq\frac{\nabla \times v_\ell}{r}$ in $E_T(L^2)$ for all $T>0$. Indeed, the solution $v_\ell \in E_T(H^2)$, given by Proposition \ref{prop.local.NS_l}, for all $T<T^*$, combined with the Biot-Svart law, (see Lemma 2.2 from \cite{Abidi}) ensure that  $\alpha$ lies in $E_T(L^2)$, for all $T<T^*$. Now, we state the proof of the $H^1-$global bound of $v_{\ell}$ in the spirit \cite{Abidi}. For the rest of the proof we will agree the following notations. 
\begin{equation*}
\Gamma_h \triangleq (\nabla \times v_h)\vec e_{\theta},\;\Gamma_{\ell} \triangleq (\nabla \times v_\ell)\vec e_\theta,\; g_\ell \triangleq( \nabla \times f_\ell )\vec e_{\theta}.
\end{equation*}
 Recall that $v=v_h+v_{\ell}$ and it is clear that $\Gamma_{\ell}$ solves
\begin{equation}\label{Gamma^l modified equation}
\partial_t \Gamma_{\ell} +\Big(v^r  \partial_r   + v^z\partial_z   - \frac{v^r}{r}\Big)\Gamma_{\ell} - \Big( \partial_{r^2}^{2} +\partial_{z^2}^{2}+\frac{\partial_{r}}{r} - \frac{1}{r^2} \Big)\Gamma_{\ell} = g_{\ell} - v_{\ell}^{r} \partial_{r} \Gamma_h-v_{\ell}^{z} \partial_{z}\Gamma_h+ \frac{v_{\ell}^r}{r} \Gamma_{h},
\end{equation} 
with $(v_{h}^{r},v_{h}^{z})$ and $(v_{\ell}^{r},v_{\ell}^{z})$ refer respectively to the components of $v_{h}$ and $v_{\ell}$.

\hspace{0.5cm}
Taking the $L^2$-inner product of \eqref{Gamma^l modified equation} by $\Gamma_\ell$, integrating by parts, the  incompressibility condition $\Div v_\ell = \partial_r v_{\ell}^r + \partial_z v_{\ell}^z  + \frac{v_{\ell}^r}{r}=0$ and $\Div v_h =\partial_r {v_h}^r + \partial_z v_{h}^z  + \frac{v_{h}^{r}}{r}=0$ yield

\begin{eqnarray*}
\frac12\frac{d}{dt} \|\Gamma_{\ell}\|_{L^2}^2+\|D \Gamma_{\ell}\|_{L^2}^2+\Big\|\frac{\Gamma_\ell}{r}\Big\|_{L^2}^2= \int_{I(\pi,\RR,\RR_+)} \Big( g_{\ell} \Gamma_{\ell}+\frac{v^r}{r} (\Gamma_{\ell} )^2 + (v_{\ell}^r \partial_r \Gamma_{\ell} + v_{\ell}^z\partial_z \Gamma_{\ell} ) \Gamma_h + \frac{v_{\ell}^r}{r} \Gamma_h \Gamma_{\ell}\Big)r drdzd\theta
\end{eqnarray*}
Via, Cauchy-Schwartz and H\"older's inequalities, we end up with  
\begin{eqnarray*}
\frac12\frac{d}{ dt}\|\Gamma_{\ell}\|_{L^2}^2+\|D \Gamma_{\ell}\|_{L^2}^2+\Big\|\frac{\Gamma_{\ell}}{r}\Big\|_{L^2}^2 &\le &\|g_\ell \|_{\dot{H}^{-1}}\|\Gamma_{\ell}\|_{\dot{H}^{ 1}} + \Big\|\frac{\Gamma_{\ell}}{r}\Big\|_{L^2} \|v^r\|_{L^6} \|\Gamma_{\ell}\|_{L^3} + \| v_{\ell}\|_{L^6}\|D \Gamma_{\ell}\|_{L^2}\|\Gamma_h\|_{L^3}\\
&&+ \Big\|\frac{\Gamma_\ell}{r}\Big\|_{L^2} \| v_\ell\|_{L^6} \|\Gamma_h\|_{L^3},
\end{eqnarray*} 

In accordance with the following Sobolev embeddings $\|g_\ell \|_{\dot{H}^{-1}}\lesssim \|f_\ell\|_{L^2}, \|\Gamma_{\ell}\|_{\dot{H}^1}\approx\|D \Gamma_{\ell}\|_{L^2}$ and $\| v_\ell\|_{L^6}\lesssim \|\nabla v\|_{L^2}$ combined with Young inequaliy give rise to,
\begin{equation*}
 \frac12\frac{d}{dt} \|\Gamma_{\ell}\|_{L^2}^2+\|D \Gamma_{\ell}\|_{L^2}^2+\Big\|\frac{\Gamma_{\ell}}{r}\Big\|_{L^2}^2  \lesssim \|f_\ell\|_{L^2}^2 + \Big\|\frac{\Gamma_\ell}{r}\Big\|_{L^2}^\frac{4}{3} \|\nabla v\|_{L^2}^\frac{4}{3}\|\Gamma_\ell\|_{L^2}^\frac{2}{3} + \|\nabla v_\ell\|_{L^2}^2\|\Gamma_h\|_{L^3}^2.
\end{equation*}

After an integration in time and Young inequality with respect to time, it holds
\begin{eqnarray}\label{gamma_l}
\nonumber \|\Gamma_\ell(t)\|_{L^2}^2+ \|D\Gamma_\ell\|_{L^2_t(L^2)}^2+\Big\|\frac{\Gamma_\ell}{r}\Big\|_{L^2_t(L^2)}^2 &\lesssim &\|\Gamma_{0,\ell}\|_{L^2}^2 +\|f_\ell \|_{L^2_t(L^2)}^2 +\Big\|\frac{\Gamma_\ell}{r}\Big\|_{L^\infty_t(L^2)}^\frac{4}{3} \|\nabla v\|_{L^2_t(L^2)}^\frac{4}{3} \|\Gamma_\ell\|_{L^2_t(L^2)}^\frac{2}{3}\\
 &&+  \int_0^t\|\nabla v_\ell(\tau)\|_{L^2}^2\|\Gamma_h(\tau)\|_{L^3}^2 d\tau.
\end{eqnarray}

Now, we treat $\alpha \triangleq \frac{\Gamma_\ell}{r}$ which is gouverns the equation
\begin{equation}\label{Gamma^l/r equation}
\partial_t \alpha  +\big(v^r  \partial_r   + v^z\partial_z \big)\alpha  - \Big( \partial_r^2 + \partial_z^2 +\frac{3}{r} \partial_r\Big) \alpha=\frac{g_\ell}{r} - \frac{v_{\ell}^r}{r} \partial_r \Gamma_h-\frac{v_{\ell}^z}{r} \partial_z\Gamma_h + \frac{v_{\ell}^r}{r} \frac{\Gamma_h}{r}.
\end{equation} 

To do so, taking the $L^2$-inner product of \eqref{Gamma^l/r equation} with $\alpha$ and we integrate over $\RR^3$, one has
\begin{eqnarray}\label{Es:princ}
\frac12\frac{d}{dt} \|\alpha\|_{L^2}^2 + \int_{\RR^3}\big(v^r  \partial_r\alpha   + v^z\partial_z\alpha \big)\alpha dx&-& \int_{\RR^3}\Big( \partial_r^2\alpha + \partial_z^2\alpha +\frac{3}{r} \partial_r\alpha\Big)\alpha dx\\
 \nonumber&\leq &\bigg| \int_{I(\pi,\mathbb{R},\mathbb{R}_+)} \Big(\frac{g_\ell}{r} - \frac{v_{\ell}^r}{r} \partial_r \Gamma_h-\frac{v_{\ell}^z}{r} \partial_z\Gamma_h + \frac{v_{\ell}^r}{r} \frac{\Gamma_h}{r}\Big)\alpha r drdzd\theta \bigg|.
\end{eqnarray}
So, for the second term of the l.h.s we have 
\begin{eqnarray*}
\int_{\RR^3}\big(v^r  \partial_r\alpha   + v^z\partial_z\alpha \big)\alpha dx&=&\int_{I(\pi,\RR,\RR_{+}}(v^r\partial_r\alpha+v^z\partial_z\alpha)\alpha rdrdzd\theta\\
&=&-\pi\int_{\RR\times\RR_{+}}\Big(\partial_r v^r+\partial_z v^z+\frac{v^r}{r}\Big)\alpha^2rdrdz=0,
\end{eqnarray*}
where we have used the fact $\Div v=\partial_r v^r+\partial_z v^z+\frac{v^r}{r}=0$ and $v^r(0,\cdot)=0$ because $v$ is axisymmetric.  

For the second term, we use the fact $\Delta =\partial_{r}^2+\partial_{z}^2+\frac1r\partial_r$ to obtain
\begin{eqnarray*}
-\int_{\RR^3}\Big( \partial_r^2\alpha + \partial_z^2\alpha +\frac{3}{r} \partial_r\alpha\Big)\alpha dx &=&-\int_{\RR^3}(\Delta\alpha)\alpha dx-\int_{\RR^3}\frac{2}{r}(\partial_r\alpha)\alpha dx\\
&=&\|\nabla\alpha\|^2_{L^2}+2\pi\int_{\RR}\alpha^2(t,0,z)dz \ge \|\nabla\alpha\|^2_{L^2}.
\end{eqnarray*}
Collecting the last two estimates and insert them in \eqref{Es:princ}, it follows
\begin{equation}\label{Eq:alphaZ}
\frac12\frac{d}{dt} \|\alpha\|_{L^2}^2 + \|D\alpha(t)\|^{2}_{L^2}\leq \bigg| \int_{I(\pi,\mathbb{R},\mathbb{R}_+)} \Big(\frac{g_\ell}{r} - \frac{v_{\ell}^r}{r} \partial_r \Gamma_h-\frac{v_{\ell}^z}{r} \partial_z\Gamma_h + \frac{v_{\ell}^r}{r} \frac{\Gamma_h}{r}\Big)\alpha r drdzd\theta \bigg|.
\end{equation}
Now, for the second term of the r.h.s, a straightforward computation yields 
\begin{equation*}
- \frac{v_{\ell}^r}{r} \partial_r \Gamma_h-\frac{v_{\ell}^z}{r} \partial_z\Gamma_h + \frac{v_{\ell}^r}{r} \frac{\Gamma_h}{r}=-v_{\ell}^r\partial_r\Big(\frac{\Gamma_h}{r}\Big)-v_{\ell}^z\partial_r\Big(\frac{\Gamma_h}{r}\Big),
\end{equation*}
so, a new use of integration by parts, the fact $\Div v_{ell}=0$ and $v_{\ell}^{r}(0,\cdot)=0$ we further get

\begin{equation*}
- \int_{I(\pi,\mathbb{R},\mathbb{R}_+)}\bigg(\frac{v_{\ell}^r}{r} \partial_r \Gamma_h+\frac{v_{\ell}^z}{r} \partial_z\Gamma_h - \frac{v_{\ell}^r}{r} \frac{\Gamma_h}{r}\bigg)\alpha rdrdzd\theta=\int_{I(\pi,\mathbb{R},\mathbb{R}_+)} \Big(v_{\ell}^r\partial_r\alpha+v_{\ell}^z\partial_r\alpha\Big)\frac{\Gamma_h}{r}rdrdzd\theta.
\end{equation*}
Gathering the last two and plug them in \eqref{Eq:alphaZ}, thus Cauchy-Schwartz's and H\"older's inequalities lead to
\begin{equation*}
\frac12\frac{d}{dt} \|\alpha\|_{L^2}^2 + \|D \alpha\|_{L^2}^2 \leq \bigg\|\frac{g_\ell}{r}\bigg\|_{L^2} \|\alpha\|_{L^2}+\|D\alpha \|_{L^2} \bigg\| \frac{v_\ell}{r}\bigg\|_{L^6} \|\Gamma_h\|_{L^3}
\end{equation*}

Integrating in time over $(0,t)$, then in view of Young inequality, one has
\begin{equation}\label{eq-alpha}
\|\alpha(t)\|_{L^2} + \|D \alpha\|_{L^2_t(L^2)}  \lesssim \|\alpha_0\|_{L^2} + \Big\|\frac{g^\ell}{r}\Big\|_{L^1_t(L^2)} + \bigg(\int_0^t     \Big\|\frac{v_\ell(\tau)}{r}\Big\|_{L^6}^2 \|\Gamma_h(\tau)\|_{L^3}^2 d\tau\bigg)^\frac{1}{2} ,
\end{equation} 

For the first term of the r.h.s., Hardy inequality, Sobolev embeddings $\dot{H}^1(\RR^3)\hookrightarrow L^6(\RR^3)$ and \eqref{Eq:2-eqv}  provide that
\begin{eqnarray*}
\Big\|\frac{v_\ell}{r}\Big\|_{L^6} \lesssim \|\partial_r v_{\ell}\|_{L^6}&\lesssim& \|\nabla v_{\ell}\|_{\dot H^1} \\
&\lesssim &\|D \Gamma_\ell\|_{L^2} + \Big\|\frac{\Gamma_\ell}{r}\Big\|_{L^2},
\end{eqnarray*}
whereas, concerning the seccond we have
\begin{equation*} 
\Big\|\frac{g_\ell}{r}\Big\|_{L^1_t(L^2)} \lesssim \|f_\ell\|_{L^1_t(H^2)}.
\end{equation*}
Plugging these into \eqref{eq-alpha}, it happens

\begin{equation}\label{eq-alpha-2}
\|\alpha(t)\|_{L^2}  + \|D \alpha\|_{L^2_t(L^2)}  \lesssim \|v_0\|_{H^2} + \|f_\ell\|_{L^1_t(H^2)} + \bigg(\int_0^t    \Big(\|D \Gamma_\ell\|_{L^2} + \bigg\|\frac{\Gamma_\ell}{r}\Big\|_{L^2}\Big)^2 \|\Gamma_h(\tau)\|_{L^3}^2d\tau\bigg)^\frac{1}{2} 
\end{equation}
\hspace{0.5cm}Let us denote 
\begin{equation*}
\Pi(T_0) \triangleq \underset{t\in[0,T_0]}{\sup} \|\Gamma_\ell(t)\|_{L^2} + \|D \Gamma_\ell\|_{L^2_t(L^2)} +\Big\|\frac{\Gamma_\ell}{r}\Big\|_{L^2_t(L^2)}.
\end{equation*}
In accordance with Sobolev embeddings and Proposition \ref{Pro:2}, one obtains
\begin{equation*}
 \|\Gamma_h \|_{L^\infty_t(L^3)} \lesssim \|v_h \|_{L^\infty_t(H^\frac{3}{2})} \lesssim \EE
\end{equation*}
Inserting this lastest in \eqref{eq-alpha-2}, then for all $t\in (0,T_0)$, we find
 \begin{equation}\label{eq-alpha-final}
\|\alpha(t)\|_{L^2}  + \|D \alpha\|_{L^2_t(L^2)}  \lesssim \|v_0\|_{H^2} + \|f_\ell\|_{L^1_t(H^2)} + \EE\Pi(T_0)
\end{equation}
Substituting \eqref{eq-alpha-final} into \eqref{gamma_l}, we infer that
\begin{eqnarray*}
\Pi(T_0)  &\lesssim & \|v_0\|_{H^1} +   \|f_{\ell}\|_{L^{2}_{T_0}(L^2)} + \Pi(T_0)\|\Gamma_h \|_{L^2_{T_0}(L^3)}  
  +  \|\alpha\|_{L^{\infty}_{T_0}(L^2)}^\frac{2}{3} \|\nabla v\|_{L^2_{T_0}(L^2)}^\frac{2}{3} \|\Gamma_\ell\|_{L^2_{T_0}(L^2)}^\frac{1}{3}\\
&  \lesssim & \|v_0\|_{H^1} + \|f_\ell\|_{L^2_{T_0}(L^2)} +  \Pi(T_0)\|\Gamma_h\|_{L^2_{T_0}(L^3)}+ \bigg( \|v_0\|_{H^2} + \|f_\ell\|_{L^1_t(H^2)}\\
&&+ \EE \Pi(T_0) \bigg)^\frac{2}{3}\times \|\nabla v\|_{L^2_{T_0}(L^2)}^\frac{2}{3} \|\nabla v_{\ell}\|_{L^2_{T_0}(L^2)}^\frac{1}{3}.
\end{eqnarray*} 
Again, Sobolev embedding and Proposition \ref{Pr:2} yield 
\begin{equation*}
\|\Gamma_h \|_{L^2_{T_0}(L^3)}  \lesssim \|v_h \|_{L^2_{T_0}(H^\frac{3}{2})} \lesssim \EE.
\end{equation*}
Now, we choose $N_0\gg1$ such that $\varepsilon\ll1$, to end with
\begin{equation*}
\Pi(T_0)\lesssim \|v_0\|_{H^1} + \|v_0\|_{H^2} + \|f_\ell\|_{L^2_{T_0}(L^2)}+\|f_\ell\|_{L^1_{T_0}(H^2)} + \|D v\|_{L^2_{T_0}(L^2)}^3.
\end{equation*}
The left hand-side of the last inequality is finite for all finite $T_0>0$, more precisely we have
\begin{equation*}
\Pi(T_0)     \lesssim \|v_0\|_{H^1} + \|v_0\|_{H^2} + \big(1 + T_0^\frac{1}{2} N_0^2\big) \|f \|_{L^2_{T_0}(L^2)} + \big(\|v_0\|_L^2 + (1+\sqrt{T_0}) \|f\|_{L^2_{T_0}(H^{-1})} \big)^3 < \infty.
\end{equation*}
The last inequality provides the desired $E_T(H^1)$, for all finite $T>0$. Theorem \ref{thm.NS-l with regular initial data} is then proved.
\end{proof}
\begin{Rema}
The $L^\beta_T(L^2)$-norm of $f$ does not explicitly apear in the last inequality above, but according to the proof of Proposition \ref{Pro:2}, in fact the choice of $N_0$ is really related to fact that $\| f \|_{L^\beta_T(L^2)}$ is finite.
\end{Rema}

\begin{proof}[Proof of theorem \ref{Th:2}]
In order to derive Theorem \ref{Th:2} from Theorem \ref{thm.NS-l with regular initial data}, we will proceed by the following argument. For $v_0$ in $H^\frac{1}{2}$, we construct a unique solution $v=v_h+v_\ell$, with $v_h$ is in $E_t(H^\frac{1}{2})$ for all $t>0$, and  $v_\ell$ solves \eqref{NS-l} in $E_{T}(H^\frac{1}{2})$, with $T<T^\star$ the maximal lifespan of existence in $H^\frac{1}{2}$. Since $v_\ell\in L^2_T(H^\frac{3}{2})$, there exists some $t_0\in (0,T^\star)$ such that $v_\ell (t_0) \in H^\frac{3}{2}$, then  Proposition \ref{prop.local.NS_l} enables to solve \eqref{NS-l} again on $(t_0,T^*)$ with initial data $v_\ell(t_0)$. So, due to the uniqueness of the solution, we infer that $v_\ell\in L^\infty((t_0,T^\star); H^\frac{3}{2}) \cap L^2((t_0,T^\star); H^\frac{5}{2})$. By the same processus there exists $t_1\in(t_{0},T^\star)$ such that $v_\ell(t_1) \in H^2$. Hence, if we take $v_\ell(t_1)$ as the new initial data we succeed to construct a unique solution given by Theorem \ref{thm.NS-l with regular initial data} on $(t_1, T)$ for all $T>t_1$, we denote this solution by $\widetilde{v}_\ell$. The uniqueness property of the Navier-Stokes equations in $E_T(\dot{H}^\frac{1}{2}) $ guarantees that $v_\ell = \widetilde{v}_\ell$ on $[t_1,T^\star)$ and Theorem \ref{Th:2} follows. 
\end{proof}
\subsection{About Boussinesq system}
The goal of this subsection is to show how to derive Theorem \ref{Th:3} from Theorem \ref{Th:2}.
\begin{proof}[Proof of Theorem \ref{Th:3}]
Since $(v_0,\rho_0)\in H^\frac{1}{2}(\RR^3) \times L^2(\RR^3)$, then in accordance of Theorem 1.1 in \cite{Danchin-Paicu1} we can construct at least one global solution $(v,\rho)$ to \eqref{B} which satisfies the energy inequalities. Next, for $\rho\in L^\infty(L^2) \subset L^\beta_{loc}(L^2)$ for $\beta> 4$, by virtue of Theorem \ref{Th:2}, we associate to this $\rho$, the unique global solution to \eqref{NS} in $E_T(H^\frac{1}{2})$, for all $T>0$. The final step consists at proving that the solution $(v,\rho)$ is the unique one. Doing so, we will need some additionnal regularity on $v$ and $\rho$, therefore the conclusion of the proof is a direct application of the following proposition.
\end{proof} 

\begin{prop}\label{prop.regularity Boussinesq}
Let $T>0$ and $(v,\rho )$ be a solution of \eqref{B} on $(0,T)$ in $E_T(H^\frac{1}{2}) \times L^\infty_T(L^2)$. If in addition $(v_0,\rho_0) \in B^0_{3,1} \times B^0_{3,1}$ then the following assertion holds
\begin{equation}\label{propagation of regularity}
(v,\rho)\in \big( L^\infty_T(B^0_{3,1}) \cap L^1_T(B^2_{3,1}) \big) \times  L^\infty_T(B^0_{3,1})
\end{equation}  
and $(v,\rho)$ is actually the unique solution to \eqref{B} on $(0,T)$.
\end{prop}

\begin{proof}
We will only show how to propagate the regularity in \eqref{propagation of regularity}, the uniqueness part is well known to hold in these spaces, see for instance theorem 1.3 in \cite{Danchin-Paicu1}.

\hspace{0.5cm}We intend to extend the proof of the previous Proposition to more general class of system than \eqref{B}. Such system takes the form
\begin{equation}\label{B-mu-general}
\left\{ \begin{array}{ll} 
\partial_{t}v-\mu\Delta v+\nabla p= Q(v,v)+\rho \vec e_3 & \textrm{if $(t,x)\in \RR_+\times\RR^3$,}\\
\partial_{t}\rho+v\cdot\nabla \rho=0 & \textrm{if $(t,x)\in \RR_+\times\RR^3$,}\\ 
\Div v=0, &\\ 
({v},{\rho})_{| t=0}=({v}_0,{\rho}_0).  
\end{array} \right.\tag{B$_{\mu,Q}$}
\end{equation}
where
\begin{equation*}
\big( Q(v,v)\big)^j= \sum_{i=1}^3q_{i,j}(D)(v^iv^j)
\end{equation*}
and $\{q_{i,j}\}_{1\leq i,j\leq 3}$ are Fourier multiplyers of order $1$.

\hspace{0.5cm}Thanks to Proposition \ref{prop1} with $f\equiv0$, we have the following estimate
\begin{equation}\label{eq.rho in B^0_{3,1}}
\|\rho\|_{\widetilde{L}^{\infty}_t( \dot{B}^0_{3,1})} \lesssim \|\rho_0\|_{\dot{B}^0_{3,1}} \bigg( 1+\int_0^t  \| \nabla v(\tau)\|_{\dot{B}^0_{\infty,1}}d\tau \bigg).
\end{equation}

Therefore, the control of $\rho$ in $ L^\infty_T(B^0_{3,1})$ requires a control of $v$ in $L^1_T(\dot{B}^1_{\infty,1})$, which can be done, due to sobolev embedding, if we know how to control $v$ in $L^1_T(B^2_{3,1}) $. To do so, we consider the Duhamel's formula associated to the velocity equation

\begin{equation*}
v(t,\cdot) = \mathbb{S}(t)v_0(\cdot) - \int_0^t \mathbb{S}(t-\tau){\bf P}Q(v , v)(\tau, \cdot) d\tau -\int_0^t \mathbb{S}(t-\tau) {\bf P} \rho \vec e_3(\tau, \cdot) d\tau.  
\end{equation*}
To make the presentation simple enough, we will use in the rest of the proof the following notation

\begin{equation*}
\widetilde{\mathcal{E}}_T(\dot{B}^0_{3,1})\triangleq \widetilde{L}^\infty_T(\dot{B}^0_{3,1}) \cap \widetilde{L}^1_T(\dot{B}^2_{3,1}).
\end{equation*}
Owing to the continuity property of the heat semi-group ${\bf S}(\cdot)$, we have

\begin{equation}\label{eq.last theorem}
\|v\|_{\widetilde{\mathcal{E}}_t(\dot{B}^0_{3,1})} \lesssim \| v_0\|_{\dot{B}^0_{3,1}}  + \|Q(v , v)\|_{\widetilde{L}^1_t(\dot{B}^0_{3,1})} + \| {\rho}\|_{\widetilde{L}^1_t(\dot{B}^0_{3,1})}.
\end{equation}  
In fact, if we deal directly with $\widetilde{L}^1_T(\dot{B}^2_{3,1}) $ estimate we will end with the following issue. By exploring the following estimate the proof of which is an easy application of paraproduct law, 

\begin{equation*}
\|Q(v,v)\|_{\widetilde{L}^1_t(\dot{B}^0_{3,1})} \lesssim \|(v , v)\|_{\widetilde{L}^1_t(\dot{B}^1_{3,1})} \lesssim \| v\|_{L^\infty_t(\dot{B}^{-1}_{\infty,\infty})}  \| v\|_{L^1_t(\dot{B}^2_{3,1})},
\end{equation*}
by denoting $\mathcal{V}(t) \triangleq \|v\|_{L^1_t(\dot{B}^2_{3,1})}$ and $K(t)\triangleq \|v\|_{L^\infty_t(H^\frac{1}{2})}$ and on account \eqref{eq.rho in B^0_{3,1}}, one has
\begin{equation}\label{first estimate V}
\mathcal{V}(t) \lesssim C_0 + K(t) \mathcal{V}(t) + \int_0^t C_0 \mathcal{V}(\tau) d\tau.
\end{equation} 

Remark, that the above estimate doesn't provide a global control of $\mathcal{V}(t)$ for all $t\leq T< \infty$. The difference factor between our context and that of \cite{Hmidi-Rousset} is the following: in \cite{Hmidi-Rousset}, the authors dealing with an initial data $v_0$ in $H^1(\RR^3)$ which gives a solution $v$ in $L^2_T(H^2)$ and eventually by interplation, one obtains in their context that $v$ belongs to $L^2_T(\dot{B}^\frac{3}{2}_{2,1})$. Or, this latter space in an algebra, whereupon the following estimate holds 

\begin{eqnarray}\label{eq.Hmidi}
\|Q (v,v)\|_{\widetilde{L}^1_t(\dot{B}^0_{3,1})} &\lesssim& \| v\|_{L^2_t(\dot{B}^{\frac{3}{2}}_{2,1})}^2\\
\nonumber&\leq& C_0 e^{C_0t}.
\end{eqnarray} 
Therefore in \cite{Hmidi-Rousset}, instead of \eqref{first estimate V}, they have
\begin{equation*}
\mathcal{V}(t) \lesssim C_0 +C_0e^{C_0t}+ \int_0^tC_0 \mathcal{V}(\tau) d\tau.
\end{equation*}
which is sufficient to control $\mathcal{V}$. Contrary, in our case we deal only with initial data in $H^\frac{1}{2}$, so the maximal gain of regularity is $L^2_T(H^\frac{3}{2})$. Consequently, we're below to the regularity threshold of Hmidi-Rousset's approach\cite{Hmidi-Rousset} which in turns becomes not available in our case. To remedy this latter, we must propose another argument, in particular to control the non-linear term without using the $L^2_t(H^\frac{3}{2})$ norm when we deal with the low frequencies in the Bony's decomposition of $v\otimes v$. The first step consists to establish a control of the velocity $v$ in $\widetilde{L}^\frac{4}{3}_t(\dot{B}^\frac{3}{2}_{3,\frac{4}{3}}) =L^\frac{4}{3}_t(\dot{B}^\frac{3}{2}_{3,{4}/{3}})$ in terms of the 
$L^4_t(\dot{H}^1)$ and $L^2_T(\dot{H}^\frac{3}{2})$ norms of $v$, which are already controlled since $v\in E_T(H^\frac{1}{2})$. In the next step, in order to control $\|v\|_{L^1_t(\dot{B}^2_{3,1})} $, we estimate the bilinear term in terms of the $L^{\frac{4}{3}}_t(\dot{B}^{\frac{3}{2}}_{3,{4}/{3}}) $ and the $L^4_t(\dot{H}^1)$ norms of $v$ , more details are given below.

\hspace{0.5cm}Taking advantage of a continuity property of $\mathbb{S}(\tau)$ we have for all $r\geq 1$
\begin{equation*}
\| \mathbb{S}(\cdot)v_0\|_{\widetilde{L}^{\frac{4}{3}}_{t}(\dot{B}^{\frac{3}{2}}_{3,r})} \lesssim \|v_0\|_{\dot{B}^0_{3,1}}.
\end{equation*}

In particular, for $r=\frac{4}{3}$, we obtain 
\begin{equation*}
\|\mathbb{S}(\cdot)v_0\|_{\widetilde{L}^{\frac{4}{3}}_{t}(\dot{B}^{\frac{3}{2}}_{3,{4}/{3}})} \lesssim \| {v_0}\|_{\dot{B}^0_{3,1}}.
\end{equation*}

Defining
\begin{equation*}
\mathcal{K}g(t,\cdot) \triangleq \int_0^t \mathbb{S}(t-\tau) g(\tau, \cdot) d\tau.
\end{equation*} 
The following inequality
   \begin{equation*}
 \|\mathcal{K} g\|_{ L^\frac{4}{3}_t (\dot{B}^\frac{3}{2}_{3,4/3})} \lesssim \| g\|_{L^\frac{4}{3}_t( \dot{B}^{\frac{3}{2}-2}_{3,{4}/{3}})}
\end{equation*}

together with Sobolev embedding $\dot{B}^{0}_{2,1}  \hookrightarrow \dot{B}^{\frac{3}{2}-2}_{3,{4}/{3}}$ and the estimate 
\begin{equation*}
 \|Q(v,v)\|_{L^\frac{4}{3}_t(\dot{B}^{0}_{2,1} )} \lesssim \| v\|_{L^4_t(\dot{H}^1)}\| v\|_{L^2_t(\dot{H}^\frac{3}{2})},
\end{equation*}
lead to
 \begin{equation*}
\|v\|_{\widetilde{L}_t^\frac{4}{3}(\dot{B}^\frac{3}{2} _{3,4/3})} \lesssim \|v_0\|_{\dot{B}^0_{3,1}}+ C\| v\|_{L^4_t(\dot{H}^1)}\| u\|_{L^2_t(\dot{H}^\frac{3}{2})}+ \|\rho\|_{L^\frac{4}{3}_t(\dot{B}^0_{3,1})}
\end{equation*}
By means of \eqref{eq.rho in B^0_{3,1}}, it follows
\begin{equation}\label{eq.L^4/3}
\|v\|_{\widetilde{L}_t^\frac{	4}{3}(\dot{B}^\frac{3}{2} _{3,4/3})} \lesssim \|v_0\|_{\dot{B}^0_{3,1}}  + C\| v\|_{L^4_t(\dot{H}^1)}\| u\|_{L^2_t(\dot{H}^\frac{3}{2})}+t^\frac{3}{4} \| \rho_0\|_{ \dot{B}^0_{3,1} }\bigg( 1+ \int_0^t \mathcal{V}(\tau) d\tau \bigg).
\end{equation}
Finally, we come back to estimate $\mathcal{V}(t)= \|v\|_{\widetilde{L}^1_t(\dot{B}^1_{3,1})}$. To do so, we summarize from \eqref{eq.last theorem}, and instead of using \eqref{eq.Hmidi}, we use the following estimates, which we will prove at the end of this section
\begin{eqnarray}\label{bilinear term-last section}
\|Q(v\otimes v)\|_{L^1_t(\dot{B}^0_{3,1})} &\lesssim& \|v\|_{L^4_t(\dot{H}^1)}\| v\|_{L^\frac{4}{3}(\dot{B}^\frac{3}{2}_{3,2})}\\
\nonumber&\lesssim& \|v\|_{L^4_t(\dot{H}^1)}\|v\|_{L^\frac{4}{3}_t(\dot{B}^\frac{3}{2}_{3,4/3})}
\end{eqnarray}
 with the help of \eqref{eq.L^4/3} and the embedding $E_T(H^\frac{1}{2}) \hookrightarrow L^4_T(\dot{H}^1)$, it happens
\begin{equation*}
\|Q(v,v)\|_{L^1(\dot{B}^0_{3,1})} \lesssim \|v\|_{E_t(H^\frac{1}{2})} \bigg( \| v_0\|_{\dot{B}^0_{3,1}}  + C\|v\|_{E_t(H^\frac{1}{2})}^2+ t^\frac{3}{4} \| \rho_0\|_{ \dot{B}^0_{3,1} }\big ( 1+ \int_0^t \mathcal{V}(\tau) d\tau \big ) \bigg).
\end{equation*}
By setting 
\begin{equation*}
\mathcal{A}(t) \triangleq \| {v_0}\|_{\dot{B}^0_{3,1}}+\| {\rho_0}\|_{\dot{B}^0_{3,1}} + \| v\|_{E_t(H^\frac{1}{2})} \big( \| {v_0}\|_{\dot{B}^0_{3,1}}  + C\| v\|_{E_t(H^\frac{1}{2})}^2 + t^\frac{3}{4}\| {\rho_0}\|_{\dot{B}^0_{3,1}} \big)
\end{equation*}
and
\begin{equation*}
\mathcal{B}(t) \triangleq \| {\rho_0}\|_{\dot{B}^0_{3,1}}\Big(  1+  t^\frac{3}{4}\| v\|_{E_t(H^\frac{1}{2})}\Big)
\end{equation*}
which are finite for all $t<\infty$. Therefore, one concludes from \eqref{eq.last theorem} that
\begin{equation*}
\mathcal{V}(t) \lesssim  \mathcal{A}(t) + \mathcal{B}(t) \int_0^t \mathcal{V}(\tau) d\tau
\end{equation*}
 Gronwall's estimate yields 
\begin{equation}\label{last-estimate}
\mathcal{V}(t) \lesssim \mathcal{A}(t)e^{t\mathcal{B}(t)}, \quad \text{ for all }\; t< \infty 
\end{equation}
Once, we have established the bound of $\mathcal{V}(t)$ for all $t>0$, we can control $v$ in $\widetilde{\mathcal{E}}_t(\dot{B}^0_{3,1})$ and $\rho$ in $\widetilde{L}^\infty_t(\dot{B}^0_{3,1}) $ by substituting \eqref{last-estimate} in \eqref{eq.last theorem} and \eqref{eq.rho in B^0_{3,1}}. Proposition \ref{prop.regularity Boussinesq} is then proved.

\hspace{0.5cm}For the sake of completeness, we briefly outline the proof of \eqref{bilinear term-last section}. For this aim, employing the following law product in three dimensions of space
\begin{equation*}
\big( \dot{H}^1 \cap \dot{B}^\frac{3}{2}_{3,2} \big) \times \big( \dot{H}^1 \cap \dot{B}^\frac{3}{2}_{3,2} \big) \hookrightarrow \dot{B}^1_{3,1}.
\end{equation*}  
Indeed, Bony's decomposition enubles us to write
\begin{equation*}
uv= T_uv + T_vu + R(u,v).
\end{equation*}
For $T_uv$, we have
\begin{align*}
\|\dot{\Delta}_j(T_uv) \|_{ L^3}& \lesssim \| \dot{S}_{j-1} u\|_{L^\infty}  \|\dot{\Delta}_j v \|_{L^3}\\
 &\lesssim c_j^2 2^{-j} \| u\|_{\dot{B}^{-\frac{1}{2}}_{\infty,2}}\| v\|_{\dot{B}^{ \frac{3}{2}}_{3,2}},
\end{align*}
with $\sum_{j\in\ZZ} c_j^2 \leq 1$. Whence, the $3D$ Sobolev embedding $\dot{H}^1 \hookrightarrow \dot{B}^{-\frac{1}{2}}_{\infty,2}$ gives
\begin{equation*}
\|T_uv \|_{\dot{B}^1_{3,1}}\lesssim \| u\|_{\dot{H}^1}\| v\|_{\dot{B}^{ \frac{3}{2}}_{3,2}}.
\end{equation*}
Similarly, by exchanging the positions of $u$ and $v$ in the above estimates, we infer that
\begin{equation*}
\|T_vu \|_{\dot{B}^1_{3,1}}\lesssim \| v\|_{\dot{H}^1}\| u\|_{\dot{B}^{ \frac{3}{2}}_{3,2}}.
\end{equation*}
Finally, for the remainder term, we proceed as follows
\begin{eqnarray*}
\|\dot{\Delta}_j(R(u,v)) \|_{ L^3}& \lesssim & \sum_{k\geq j-N_0} \| \tilde{\dot{\Delta}}_k u\|_{L^\infty}  \|\dot{\Delta}_k v \|_{L^3}\\
 &\lesssim&  2^{-j} \bigg(\sum_{k\geq j-N_0}  2^{j-k}c_k^2 \bigg) \| u\|_{\dot{B}^{-\frac{1}{2}}_{\infty,2}}\| v\|_{\dot{B}^{ \frac{3}{2}}_{3,2}},\\
 &\lesssim& 2^{-j}d_j   \| u\|_{\dot{H}^1}\| v\|_{\dot{B}^{ \frac{3}{2}}_{3,2}},\quad \sum_{j\in\ZZ} d_j \leq 1. 
\end{eqnarray*}
\end{proof}
\begin{Rema}
The uniqueness, part of Proposition \ref{prop.regularity Boussinesq} can be established as well for the general system \eqref{B-mu-general} by following the proof's approach of Theorem 1.3 from \cite{Danchin-Paicu1}, and by estimating the difference between two solutions still noted by $(\delta v, \delta \rho)$ in $\dot{B}^{-1}_{3,1} \times \dot{B}^{-1}_{3,1}$, one may notice that the proof of the uniqueness in Theorem 1.3 in \cite{Danchin-Paicu1} doesn't use at any step the special structure of the bilinear term.
\end{Rema}
 
 \section*{Acknowledgements}
 This work has been done while the first author is a PhD student at the University of Cote d'Azur-Nice-France, under the supervision of F. Planchon and P. Dreyfuss, in particular the first author would like to thank his supervisers, and the LJAD direction.\\

\end{document}